\documentclass[a4paper,11pt]{article}

\usepackage{amsmath,amsthm}
\usepackage{amssymb}
\usepackage{breqn}
\usepackage{enumerate}
\usepackage{graphicx}
\usepackage{bm}
\usepackage{bbm}
\usepackage[affil-it]{authblk}
\usepackage{tabu}
\usepackage{bold-extra}
\usepackage[hang,flushmargin]{footmisc}
\usepackage{hyperref}
\usepackage{mathtools}
\usepackage{relsize}
\usepackage{scalerel}
\usepackage{xcolor}
\usepackage{titlesec}
\usepackage{apptools}
\usepackage{appendix}

\newtheorem{theorem}{Theorem}
\newtheorem{corollary}[theorem]{Corollary}
\newtheorem{lemma}[theorem]{Lemma}
\newtheorem{proposition}[theorem]{Proposition}

\newtheorem{claim}[theorem]{Claim}

\theoremstyle{definition}
\newtheorem{defin}[theorem]{Definition}

\AtAppendix{\counterwithin{theorem}{section}}

\newcommand{\bal}{\mathsf{b}}

\newcommand{\vd}{\alpha_{\textnormal{vert}}}

\titleformat{\section}[hang]{\scshape\large\bfseries\filcenter}{\S\thesection}{4pt}{}
\titleformat{\subsection}[hang]{\scshape\bfseries}{\thesubsection.}{4pt}{}

\allowdisplaybreaks

\newcommand{\tss}[1]{\textsuperscript{#1}}

\newcommand{\on}[1]{
	\operatorname{#1}
}

\newcommand{\tdt}{\times\cdots\times}

\newcommand{\tightoverset}[2]{
  \mathop{#2}\limits^{\vbox to -.5ex{\kern-1.15ex\hbox{$#1$}\vss}}}

\newcommand\restr[2]{{
  \left.\kern-\nulldelimiterspace 
  #1 
  \vphantom{\big|} 
  \right|_{#2} 
}}

\makeatletter
\newcommand{\subalign}[1]{%
  \vcenter{%
    \Let@ \restore@math@cr \default@tag
    \baselineskip\fontdimen10 \scriptfont\tw@
    \advance\baselineskip\fontdimen12 \scriptfont\tw@
    \lineskip\thr@@\fontdimen8 \scriptfont\thr@@
    \lineskiplimit\lineskip
    \ialign{\hfil$\m@th\scriptstyle##$&$\m@th\scriptstyle{}##$\hfil\crcr
      #1\crcr
    }%
  }%
}
\makeatother

\newcommand\blfootnote[1]{%
  \begingroup
  \renewcommand\thefootnote{}\footnote{#1}%
  \addtocounter{footnote}{-1}%
  \endgroup
}

\newcommand\ssk[1]{
	\substack{#1}
}

\newcommand\ex{\mathop{\mathbb{E}}}

\newcommand{\exx}{
  \mathop{
    \mathchoice{\vcenter{\hbox{\larger[4]$\mathbb{E}$}}}
               {\kern0pt\mathbb{E}}
               {\kern0pt\mathbb{E}}
               {\kern0pt\mathbb{E}}
  }\displaylimits
}

\makeatletter
\newcommand*\bcdot{\mathpalette\bigcdot@{0.5}}
\newcommand*\bigcdot@[2]{\mathbin{\vcenter{\hbox{\scalebox{#2}{$\m@th#1\bullet$}}}}}
\makeatother

\makeatletter
\def\blfootnote{\gdef\@thefnmark{}\@footnotetext}
\makeatother

\newcommand\id{\mathbbm{1}}

\begin{document}
\begin{center}\Large\noindent{\bfseries{\scshape Good bounds for sets lacking skew corners}}\\[24pt]\normalsize\noindent{\scshape Luka Mili\'cevi\'c\dag}\\[6pt]
\end{center}
\blfootnote{\noindent\dag\ Mathematical Institute of the Serbian Academy of Sciences and Arts\\\phantom{\dag\ }Email: luka.milicevic@turing.mi.sanu.ac.rs}

\footnotesize
\begin{changemargin}{1in}{1in}
\centerline{\sc{\textbf{Abstract}}}
\phantom{a}\hspace{12pt}~A skew corner is a triple of points in $\mathbb{Z} \times \mathbb{Z}$ of the form $(x,y), (x, y + a)$ and $(x + a, y')$. Pratt posed the following question: how large can a set $A \subseteq [n] \times [n]$ be, provided it contains no non-trivial skew corner (i.e. one for which $a\not=0$)? We prove that $|A| \leq \exp(- c\log^c n) n^2$, for an absolute constant $c > 0$, which, along with a construction of Beker, essentially resolves Pratt's question.\\ 
\phantom{a}\hspace{12pt}~Our argument is represents a two-dimensional variant of the method of Kelley and Meka, which they used to prove Behrend-type bounds in Roth's theorem. A very similar result was obtained independently and simultaneously by Jaber, Lovett and Ostuni.
\end{changemargin}
\normalsize

\section{Introduction}

\hspace{12pt} In~\cite{Pratt}, Pratt defined \textit{skew corner} to be a configuration of the form $(x, y), (x, y + a),(x+ a, y')$, with $x, y, y', a\in \mathbb{Z}$. If $a \not= 0$, we say that the skew corner is \textit{non-trivial}. In that paper, as a part of study of certain group-theoretic approaches to the problem of obtaining an optimal matrix multiplication algorithm originating in~\cite{CKSUmatrix, CUmatrix}, Pratt posed the following problem: how large can a subset $A \subseteq [n] \times [n]$ be, provided it contains no non-trivial skew corners? Let us write $s(n)$ for this quantity. Additionally, he conjectured that in fact $s(n) = O_\varepsilon(n^{1 + \varepsilon})$ for every $\varepsilon > 0$. In particular, if it had held, that conjecture would have implied that the approaches above cannot give an optimal matrix multiplication algorithm.\\
\indent Notice that, if one sets $y' = y$, then a skew corner becomes a corner in the usual sense in the context of additive combinatorics. Finding bounds on the size of sets that lack non-trivial corners is a well-studied problem, with Ajtai and Szemer\'edi~\cite{AjtaiSzem} obtaining first bounds of the form $o(n^2)$. The best known bounds in the latter problem are due to Shkredov~\cite{Shkredov}, immediately giving $s(n) \leq O(n^2 / (\log \log n)^{\Omega(1)})$. When it comes to lower bounds, Petrov~\cite{Petrov} showed that $s(n) \geq \Omega(n \log n/\sqrt{\log \log n})$. In~\cite{PohoataZakharov}, Pohoata and Zakharov showed $s(n) \geq \Omega(n^{5/4})$, thus refuting Pratt's conjecture. They even conjectured that $s(n) \leq O(n^{2-c})$ for an absolute constant $c > 0$. Very recently, Beker~\cite{Beker} disproved their conjecture by obtaining a Behrend-type lower bound and also improved the upper bounds, namely
\begin{equation}\label{BekerBound}n^2 \exp(-O(\sqrt{n})) \leq s(n) \leq O(n^2\log^{-\Omega(1)} n).\end{equation}
\indent Finally, Beker asked whether Behrend-shape upper bounds hold for the skew corners-free sets. Our main result is that this is indeed the case, which essentially answers the initial question of Pratt, at least when it comes to the shape of the bounds. 

\begin{theorem}\label{mainthm} There exists an  absolute constant $c > 0$ with the following property. Let $A \subseteq [n] \times [n]$ be a skew corner-free set, i.e. a set which contains no configuration of the form $(x, y), (x, y + a),(x+ a, y')$, with $x, y, y', a\in \mathbb{Z}$ and $a \not =0$. Then 
\[|A| \leq \exp(- c\log^c n) n^2.\]
\end{theorem}

\noindent\textbf{Remark.} A very similar result was obtained independently and simultaneously by Jaber, Lovett and Ostuni~\cite{JLOskew}.\\

Our approach is closely related to the breakthrough work of Kelley and Meka~\cite{KelleyMeka}, where they proved Behrend-shape bounds in Roth's theorem. However, given that the problem considered in this paper is two-dimensional rather than one-dimensional leads to important differences. Let us also mention that Peluse posed the problem of improving the bounds on the corner-free sets using the methods of Kelley and Meka (see Problem 1.18 in~\cite{PeluseSurvey} and the surrounding discussion). Hence, the work in this paper can be seen as progress towards that goal.

\subsection{Proof overview}

 We include a very brief summary of Kelley and Meka's proof in this subsection (see also an exposition by Bloom and Sisask~\cite{BloomSisask}), but before we embark on that discussion, let us explore the traditional additive-combinatorial approaches to questions of finding arithmetic configurations.\\

\indent For the following discussion, we use the usual notion of \textit{discrete multiplicative derivative} $\partial_a$, which, for a given function $f \colon G \to \mathbb{C}$ gives another one defined as $\partial_a f(x) = f(x + a) \overline{f(x)}$. A fundamental tool in the context of counting arithmetic configurations are uniformity norms, introduced by Gowers in~\cite{Gow1,Gow2}.  For example, if $f_1, f_2, f_3 \colon G \to \mathbb{D}$ are three functions, in the case of arithmetic progressions of length 3, the uniformity norm $\mathsf{U}^2$, defined via formula $\|f\|_{\mathsf{U}^2}^4 = (\ex_{a,b,x} \partial_a \partial_b f(x))^{1/4}$, controls the corresponding counting multilinear form via inequality $|\ex_{x, a} f_1(x) f_2(x + a) f_3(x + 2a)| \leq \|f_3\|_{\mathsf{U}^2}.$ The proof proceeds via Cauchy-Schwarz inequality (abbreviated to CS below)
\begin{align}
    \Big|\exx_{x, a} f_1(x) &f_2(x + a) f_3(x + 2a)\Big|^4 = \Big|\exx_{x} f_1(x) \Big(\exx_a f_2(x + a) f_3(x + 2a)\Big)\Big|^4\nonumber\\
    \leq &\Big|\exx_{x} |f_1(x)|^2 \Big|^2 \Big|\exx_x\Big|\exx_a f_2(x + a) f_3(x + 2a)\Big|^2\Big|^2 \hspace{2cm}\text{(using CS)}\nonumber\\
    \leq &\Big|\exx_x\Big|\exx_a f_2(x + a) f_3(x + 2a)\Big|^2\Big|^2\label{elimStep1}\\
    = &\Big|\exx_{x, a, a'}f_2(x + a) f_3(x + 2a) \overline{f_2(x + a') f_3(x + 2a')}\Big|^2 \hspace{2cm}\text{(after expanding)}\nonumber\\
    = &\Big|\exx_{x, a, y} \partial_a f_2(x) \partial_{2a} f_3(x + y)\Big|^2\hspace{2cm}\text{(change of variables)}\nonumber\\
    = & \Big|\exx_{x, a} \partial_a f_2(x) \Big(\exx_y \partial_{2a} f_3(x + y)\Big)\Big|^2\nonumber\\
     \leq &\Big(\exx_{x, a} |\partial_a f_2(x)|^2 \Big) \Big(\exx_{x, a} \Big|\exx_y \partial_{2a} f_3(x + y)\Big|^2\Big)  \hspace{2cm}\text{(using CS)}\nonumber\\
     \leq &\exx_{x, a} \Big|\exx_y \partial_{2a} f_3(x + y)\Big|^2\label{elimStep2}\\
     = &\exx_{x, a, b}  \partial_{2a, b} f_3(x).\nonumber\end{align}

Observe that each Cauchy-Schwarz step gives rise to another application of discrete multiplicative derivative. The derivative is applied in the direction corresponding to the difference between corresponding points in the arguments of the functions. More precisely, notice that in each Cauchy-Schwarz step we eliminate a term, say $f(x)$, e.g. in~\eqref{elimStep1} we eliminated term involving $f_1$ and in~\eqref{elimStep2} we eliminated term involving $f_2$. Then every remaining term $g(y)$ receives additional derivative $\partial_{y - x}$ upon expanding in the next step.\\
\indent We may carry out analogous steps in the case when we count skew corners. This time the method above gives the following inequalities, depending on which variables we use in outer expectation in the applications of Cauchy-Schwarz inequality,\footnote{We are slightly more careful in the last step of the method sketched above; that step actually gives upper bound which is the product of norms of two functions, rather than a norm of a single function.}
\[\Big|\exx_{x, y, a, y'} f_1(x,y) f_2(x, y + a) f_3(x +a, y')\Big| \leq \|f_i\|_{\mathsf{U}(G \times G, \{0\} \times G))} \|f_3\|_{\mathsf{U}^2},\]
for $i \in \{1,2\}$, and
\[\Big|\exx_{x, y, a, y'} f_1(x,y) f_2(x, y + a) f_3(x +a, y')\Big| \leq  \|f_1\|_{\mathsf{U}(G \times G, \{0\} \times G))} \|f_2\|_{\mathsf{U}(G \times G, \{0\} \times G))},\]
and where $\|\cdot\|_{\mathsf{U}(H_1, H_2)}$ is the directional uniformity norm\footnote{See \cite{Austin1, Austin2, MilDir} for the definition and discussion of directional uniformity norms.} with derivative groups $H_1, H_2 \leq G \times G$, i.e.
\[\|f\|_{\mathsf{U}(H_1, H_2)} = \exx_{x \in G \times G, a_1 \in H_1, a_2 \in H_2} \partial_{a_1}\partial_{a_2}f(x).\]

In particular, having the $\mathsf{U}^2$ control is sufficient when it comes to obtaining quantitative bounds for this problem. This observation was exploited by Beker~\cite{Beker} in his proof of the upper bound~\eqref{BekerBound}, (at least in the finite vector space model of the problem; see Proposition 3.1 in that paper).\\

\indent However, in order to achieve Behrend-type bounds, according to Kelley-Meka method, it is crucial that we have a \textit{simultaneous} control over the two norms appearing in the final inequalities in the method above. In a nutshell, in the Kelley-Meka method we replace the first application of Cauchy-Schwarz inequality by H\"older's inequality and in then their method can be viewed as weak form of the inverse theory for the resulting norm. But, if we were to have two applications H\"older's inequality, then the possible inverse conjectures would not necessarily imply the desired density increment. See Example 8 and surrounding discussion in the paper of Shkredov~\cite{ShkredovKM}. So, in some sense, our work is adaptation of Kelley-Meka method for the directional norm ${\mathsf{U}(G \times G, \{0\} \times G))}$.\\

Furthermore, the problem of finding skew corners has an additional level of difficulty in comparison with Roth's theorem. Namely, in Roth's theorem, the only structured sets that play a role, stemming from the obstruction to uniformity for $\mathsf{U}^2$ norm, are precisely (coset of) subspaces of $G$. Thus, in that problem, we have the following simple hierarchy of structures, $A' \subseteq t + H \subseteq G$, where $A'$ is the intersection $A \cap (t + H)$, for the given set $A$, and $t + H$ is a coset of a subspace that arises during the density increment procedure. If $A'$ is quasirandom subset of $t+H$ in the $\mathsf{U}^2$ sense, this hierarchy completely determines the count of arithmetic progressions of length 3, and the density of $t + H$ in $G$ and the relative density of $A'$ in $t + H$ have different effects on this count.\\
\indent When it comes to finding skew corners, we need an additional kind of structure in the hierarchy, namely sets that depend only on the first coordinate. This time, the hierarchy is $A' \subseteq X \times (t + H) \subseteq (s + H) \times (t + H) \subseteq G \times G$, which has more levels than that required in the Roth's theorem.\\

Let us now return to an outline of our work and examine more closely Kelley and Meka's approach.\\
\indent Let us write $\|\cdot\|_{\mathsf{U}^{2,r}_{\textnormal{KM}}}$ for the norm arising in their work. Namely,
\[\|f\|_{\mathsf{U}^{2,r}_{\textnormal{KM}}} = \bigg(\exx_{\ssk{a_1, a_2\\b_1, \dots, b_r}} \prod_{i \in [r]} f(x + a_1 + b_i) f(x + a_2 + b_i)\bigg)^{1/2r}.\]
Note that in the case $r = 2$, this quantity becomes the usual $\mathsf{U}^2$ uniformity norm.\\
\indent Even though this paper is exclusively about a problem in cyclic groups, to simplify the notation in this introductory discussion, in the rest of this section (with the exception of Corollary~\ref{nonbinsys}), we consider the finite vector space and write $G = \mathbb{F}_p^n$. Of course, in the actual work we shall use appropriate substitutes for various notions appearing here, e.g. Bohr sets for subgroups etc.\\
\indent Let $A \subseteq G$ be a set of density $\delta$ without arithmetic progression of length 3. Let us write $\bal_A = \id_A - \delta$ for the \textit{balanced function} of $A$. In the simplest terms, the argument of Kelley and Meka consists of the following steps.
\begin{itemize}
    \item[\textbf{Step 1.}] Starting from the assumption that $A$ has no arithmetic progressions, they obtain $\|\bal_A\|_{\mathsf{U}^{2,r}_{\textnormal{KM}}} \geq \Omega(\delta)$ for a suitable $r$. Since these norms quantify the deviation of $A$ from a quasirandom set, we may say that the aim of this step is to \textbf{find imbalance in $A$}. (This step was called \textit{H\"older lifting} in~\cite{BloomSisask}.)
    \item[\textbf{Step 2.}] In the second step, it is shown that the inequality $\|\bal_A\|_{\mathsf{U}^{2,r}_{\textnormal{KM}}} \geq \Omega(\delta)$ implies $\|\bal_A\|_{\mathsf{U}^{2,r'}_{\textnormal{KM}}} \geq 1 + \Omega(\delta)$ for some $r'$ which is slightly larger than $r$. So in this step they \textbf{pass from $\bal_A$ to $\id_A$}. Kelley and Meka called his step \textit{spectral positivity} (and in~\cite{BloomSisask} it is referred to as \textit{unbalancing}).
    \item[\textbf{Step 3.}] Starting from inequality $\|\bal_A\|_{\mathsf{U}^{2,r'}_{\textnormal{KM}}} \geq 1 + \Omega(\delta)$, using dependant random choice, in the step called \textbf{sifting}, they find an additional set $A'$, of density $\alpha'$, such that
    \[\exx_{x, y} \id_S(x)  \id_{A'}(x + y) \id_{A'}(y) \geq (1 + \Omega(1)) {\alpha'}^2,\]
    where $S$ is the set of elements $x$ where the self-convolution $\id_A \ast \id_A$ is significantly larger than its mean $\delta^2$.
    \item[\textbf{Step 4.}] \textbf{Using almost-periodicity results} stemming from the works of Croot and Sisask, and Sanders, applied to the expression above, they obtain a subspace $V$ of small codimension such that
    \[\exx_{x, y \in G, v \in V} \id_S(x + v)  \id_{A'}(x + y) \id_{A'}(y) \geq (1 + \Omega(1)) {\alpha'}^2.\]
    \item[\textbf{Step 5.}] The final step is to obtain an efficient density increment using the conclusions above.
\end{itemize}

In this work, however, we are looking for skew corners rather than arithmetic progressions of length 3, giving rise to a different type of structural hierarchy, as mentioned above. This has two important ramifications. Firstly, throughout the proof, instead of norm $\|\cdot\|_{\mathsf{U}^{2,r}_{\textnormal{KM}}}$, we analyse the norm called \textit{vertical segments norm}, which is defined as 
\[\|f\|_{\mathsf{VS}_r} = \Big(\exx_{x_1, \dots, x_r,  y_1, \dots, y_r, a} \prod_{i \in [r]} f(x_i, y_i) f(x_i, y_i + a)\Big)^{1/2r},\]
for $f \colon G \times G \to \mathbb{R}$. Their name stems from the fact that the points in the arguments of $f$ in the expression above belong to $r$ vertical segments of length $a$.\\

This is a good place to mention another variant of such norms, called \textit{grid norms}, defined as
\[\|f\|_{\mathsf{gr}_{2, r}} = \Big(\exx_{x_1, x_2 \in X, y_1, \dots, y_r \in Y} \prod_{i \in [r]} f(x_1, y_i) f(x_2, y_i)\Big)^{1/2r},\]
for functions $f \colon X \times Y \to \mathbb{R}$, which were introduced by Kelley, Lovett and Meka in~\cite{KLMgrid}. Just like $\|\cdot\|_{\mathsf{U}^{2,r}_{\textnormal{KM}}}$ corresponds to the usual $\mathsf{U}^2$ norm, and vertical segments norm corresponds to the directional norm $\|\cdot\|_{\mathsf{U}(G \times G, \{0\} \times G)}$, the grid norms correspond to 2-dimensional box norm.\\

Given that we work with different norms, we need to modify the steps of Kelley and Meka's proof. In particular, we can no longer use the spectral positivity step, as that argument relied on the specific form of the norm $\|\cdot\|_{\mathsf{U}^{2,r}_{\textnormal{KM}}}$. However, in~\cite{FHHK}, where they generalize the argument of Kelley and Meka to certain systems of linear forms, Filmus, Hatami, Hosseini and Kelman encounter a similar difficulty, as they consider grid norms instead of $\|\cdot\|_{\mathsf{U}^{2,r}_{\textnormal{KM}}}$, and their argument for this step is based on a combination a more direct use of a variant of Gowers-Cauchy-Schwarz inequalities and an inequality of Kelley and Meka concerning functions with non-negative odd moments (see Lemma~\ref{binomialIneq}). We overcome this issue in a similar spirit in our paper.\\  

Furthermore, as our structural hierarchy is $A \subseteq X \times (t + H) \subseteq (s + H) \times (t + H) \subseteq G \times G$, we additionally have to perform a regularization step at the beginning of the density increment strategy, which ensures that columns of $A$ are well-behaved, as well as the set $X$ itself. Hence, our proof consists of the following steps.
\begin{itemize}
    \item[\textbf{Step 1.}] \textbf{Regularization of $A$.} In this step we pass to well-behaved columns of the set $A$.
    \item[\textbf{Step 2.}] \textbf{Finding imbalance in $A$.} We show that a suitably defined balanced function of $A$ is not quasirandom, i.e. that it has large vertical segments norm.
    \item[\textbf{Step 3.}] \textbf{Passing from $\bal$ to $\id_A$.} We conclude that the vertical segments norm of $\id_A$ is significantly higher than that of a randomly chosen subset of $X \times (t + H)$ of the same density.
    \item[\textbf{Step 4.}] \textbf{Sifting.} We use a dependent random choice argument to find sets to turn the information on the vertical segments norm of $\id_A$ into further counts of arithmetic configurations.
    \item[\textbf{Step 5.}] \textbf{Using almost-periodicity.} We use almost-periodicity results to find strong algebraic structure. Let us remark that, even though the problem at hand is two-dimensional, we may again use one-dimensional almost-periodicity results. 
    \item[\textbf{Step 6.}] \textbf{Completing the proof.} The conclusions of previous steps are put together to find a desired density increment.
\end{itemize}

Our main technical result is Proposition~\ref{densincprop}, which sums up a single density increment step.\\

We conclude this section with an additional comment on how our results differ from those of Filmus, Hatami, Hosseini and Kelman~\cite{FHHK}, apart from the obvious distinction that their main arithmetic result is about one-dimensional structures. In that work, they prove a general counting lemma for graphs (Theorem 1.11 in their paper), which in the conclusion has a density increment on a product of unstructured sets $S \times T$, which is unsuitable for our purposes, as we look for structure of the form $S \times (t + H)$ for a subgroup $H \leq G$.\\
\indent On the other hand, their main result on arithmetic configurations concerns (special instances of) affine \textit{binary systems of linear forms}. These are configurations parameterized by variables $x_1, \dots, x_k$, where each point is $\lambda_{i j_1} x_{j_1} + \lambda_{i j_2} x_{j_2}$, $\lambda_{i j_1}, \lambda_{i j_2} \not= 0$ and $j_1 \not= j_2$, with the crucial requirement that pairs $\{j_1, j_2\}$ are distinct for every $i$. For example, arithmetic progression of length 3 can be parametrized as $\Big(x-y, \frac{x - z}{2}, y-z\Big)$, which is in the described family. We now give a corollary of Theorem~\ref{mainthm} in the 1-dimensional setting, which cannot be realized as a binary system of linear forms.

\begin{corollary}\label{nonbinsys}
    There exists an absolute constant $c' > 0$ with the following property. Let $B \subseteq [n]$ be a set that has no non-trivial configurations of the form
    \[\Big(x, x+y, x +2y, x+ y + a, x + 2y + 2a, x + a\Big)\]
    for some $x, y, a \in \mathbb{Z}$. Then 
    \[|B| \leq \exp(- c'\log^{c'} n) n.\]
\end{corollary}

\begin{proof}[Proof sketch] Write $\beta = |B| /n$. Let $A \subseteq [n] \times [n]$ consist of all $(x,y)$ such that $x, x + y, x + 2y \in B$. By Kelley and Meka's proof of Roth's theorem, we see that $A$ has density $\exp(- O(\log^{O(1)} \beta^{-1}))$ inside $[n] \times [n]$. It follows from Theorem~\ref{mainthm} that, unless $\beta$ is smaller than claimed bound, $A$ has a non-trivial skew corner, which can be reinterpreted as the configuration in the statement.\end{proof}

Let us briefly indicate why this configuration cannot be realized as a binary system of linear forms. Suppose that we did have such a parametrization using additional variables $u_1, \dots, u_k$. Let us write $x(u), y(u), a(u)$ for resulting elements $x, y, a$. Observe equalities
\[x+ y + a = (x + y) + (x + a) - x\]
and
\[x + 2y + 2a = 2(x + y) + 2(x + a) - 3x.\]
However, in a binary system, such equalities completely determine the two variables $u_i$ which are used in the linear form on the left hand side, implying that forms $(x+ y + a)(u)$ and $(x + 2y + 2a)(u)$ have the same pair of parametrizing variables.\\
\indent The configuration in the corollary can be see as a union of two arithmetic progressions of length 3, starting at the same base point $x$, with steps $a$ and $b$, with an additional point $x + a - b$. Without the last point, it is easy to obtain such a 5-point configuration from Roth's theorem by Cauchy-Schwarz inequality (in fact, such a system is an affine binary systems of linear forms). However, the introduction of the last point adds another non-trivial arithmetic restriction and seems to make the problem more difficult.\\

\noindent\textbf{Acknowledgements.} This research was supported by the Ministry of Science, Technological Development and Innovation of the Republic of Serbia through the Mathematical Institute of the Serbian Academy of Sciences and Arts, and by the Science Fund of the Republic of Serbia, Grant No.\ 11143, \textit{Approximate Algebraic Structures of Higher Order: Theory, Quantitative Aspects and Applications} - A-PLUS. 

\section{Preliminaries}

In this preliminary section, we gather a few general technical results that we will exploit later in the paper. We use the standard averaging notation $\ex_{x \in X}$, which is the shorthand for $\frac{1}{|X|}\sum_{x \in X}$, and we also write $\mathbb{D} = \{z \in \mathbb{C} \colon |z|  =1 \}$ for the unit disk.

\subsection{Bohr sets}

As it is well-known, due to lack of subgroups in cyclic groups $\mathbb{Z}/N\mathbb{Z}$ we are forced to work with approximate substitutes of subgroups. With this in mind, as it is standard in additive combinatorics, we use the machinery of (regular) Bohr sets, introduced by Bourgain in~\cite{Bourgain}. In the precise statements of defnitions below, we follow the work of Schoen and Sisask~\cite{SchSisask}, in which they prove Behrend-type bounds for sets lacking solutions to 4 variable equations, as we need their variants of almost-periodicity results.

\begin{defin}Let $\Gamma \subseteq \mathbb{Z}/N\mathbb{Z}$ be a set and let $\rho \geq 0$. The \emph{Bohr set} on the \emph{frequency set} $\Gamma$ with \emph{radius} $\rho$ is given by
\[\on{Bohr}(\Gamma, \rho) = \{x \in \mathbb{Z}/N\mathbb{Z} \,\colon (\forall r \in \Gamma)\,\,|\on{e}(r x / N) - 1| \leq \rho\},\]
where $\on{e}(t) = \exp(2  \pi i t)$. We refer to $|\Gamma|$ as the \emph{rank} of the Bohr set.\\
\indent If $B$ is a Bohr set given by $B = \on{Bohr}(\Gamma, \rho)$, for a real $\delta \geq 0$, we define its \emph{$\delta$-dilate} $B_\delta$ as $\on{Bohr}(\Gamma, \delta\rho)$.\\
\indent A Bohr set $B$ of rank $d$ is \emph{regular} if 
\[1-12d|\delta| \leq \frac{|B_{1 + \delta}|}{|B|} \leq 1 + 12d|\delta|\]
whenever $|\delta| \leq 1/12d$.
\end{defin}

\begin{lemma}[Elementary properties of Bohr sets] \label{basicBohr}Suppose that $B = B(\Gamma, \rho)$ is a Bohr set of rank $d$.
\begin{itemize}
\item[\textbf{(i)}] For any $\delta, \delta' \geq 0$ we have $B_\delta + B_{\delta'} \subseteq B_{\delta + \delta'}$.
\item[\textbf{(ii)}] We have the \emph{size estimate} $|B| \geq (\rho/2\pi)^d N$. 
\item[\textbf{(iii)}] We have the \emph{doubling estimate} $|B_2| \leq 6^d |B|$.
\item[\textbf{(iv)}] There exists a $\delta \in [1/2,1]$ for which $B_\delta$ is regular.
\end{itemize}
\end{lemma}

We need a general lemma that allows us to carry out linear changes of variables. In contrast to the finite vector space case, we are only allowed to modify variables by other ones that range over narrower Bohr sets. Firstly, we prove the case when we introduce a single dummy variable.\\

\begin{lemma}[Basic change of variables] \label{chVarLemmaBasic}
Let $B^{(1)}, \dots, B^{(k)}$ be regular Bohr sets of rank at most $d$, and let $S \subseteq \cap_{i \in [k]} B^{(i)}_\rho$. Let $F \colon (\mathbb{Z}/N\mathbb{Z})^{k+1} \to \mathbb{D}$ be a function. Let $\nu_1, \dots, \nu_k \in \mathbb{Z}$ be given. Then
\begin{align*}&\bigg|\exx_{x_1 \in B^{(1)}, \dots, x_k \in B^{(k)}, y \in S} F(x_1, \dots, x_k, y)  - \exx_{x_1 \in B^{(1)}, \dots, x_k \in B^{(k)}, y \in S} F(x_1 + \nu_1 y, \dots, x_k + \nu_k y, y)\bigg| \leq  50 \rho k d \|\nu\|_\infty.\end{align*}
\end{lemma}

\vspace{\baselineskip}

In this work, we shall typically apply the lemma above to the case where
\[F(x_1, \dots, x_k, y) = \prod_{i \in [s]} f_i\Big(\sum_{j \in [k]} \lambda_{i j} x_j + \lambda_{i k+1}y, \sum_{j \in [k]} \mu_{i j} x_j + \mu_{i k+1}y\Big)\]
for some functions $f_1, \dots, f_s \colon \mathbb{Z}/N\mathbb{Z} \times \mathbb{Z}/N\mathbb{Z} \to \mathbb{D}$ and $\lambda_1, \dots, \lambda_s, \mu_1, \dots, \mu_s \in \mathbb{Z}^{k+1}$.\\

\begin{proof}Fix arbitrary $y \in S$. For the given $y$ we prove the related inequality
\[\bigg|\exx_{x_1 \in B^{(1)}, \dots, x_k \in B^{(k)}} F(x_1, \dots, x_k, y)  - \exx_{x_1 \in B^{(1)}, \dots, x_k \in B^{(k)}} F(x_1 + \nu_1 y, \dots, x_k + \nu_k y, y)\bigg| \leq  50 \rho k d \|\nu\|_\infty\]
and then the lemma follows from the triangle inequality. Let $K = \|\nu\|_\infty$. Consider the set $X = B^{(1)}_{1 - K\rho} \tdt B^{(k)}_{1 - K\rho}$. Note that $(x_1 + \nu_1 y, \dots, x_k + \nu_k y) \in B^{(1)} \tdt B^{(k)}$ for all $(x_1, \dots, x_k) \in X$. Therefore, for all elements $(x_1, \dots, x_k) \in X \cap (\nu_1 y, \dots, \nu_k y) + X$ the term $F(x_1, \dots, x_k, y)$ appears in both of averages
\[\exx_{x_1 \in B^{(1)}, \dots, x_k \in B^{(k)}} F(x_1, \dots, x_k, y)\,\,\text{ and }\,\,\exx_{x_1 \in B^{(1)}, \dots, x_k \in B^{(k)}} F(x_1 + \nu_1 y, \dots, x_k + \nu_k y, y).\] 
Hence, writing $\tilde{B} = B^{(1)} \tdt B^{(k)}$, for which we know that $X \pm (\nu_1 y, \dots, \nu_k y) \subseteq \tilde{B}$, we see that
\begin{align*}&\bigg|\exx_{x_1 \in B^{(1)}, \dots, x_k \in B^{(k)}} F(x_1, \dots, x_k, y)  - \exx_{x_1 \in B^{(1)}, \dots, x_k \in B^{(k)}} F(x_1 + \nu_1 y, \dots, x_k + \nu_k y, y)\bigg|\\
&\hspace{2cm}\leq \frac{|\tilde{B} \setminus (X \cap (\nu_1 y, \dots, \nu_k y) + X)| + |(\tilde{B} + (\nu_1 y, \dots, \nu_k y)) \setminus (X \cap (\nu_1 y, \dots, \nu_k y) + X)|}{|\tilde{B}|}\\
&\hspace{2cm}\leq \frac{4 |\tilde{B} \setminus X|}{|\tilde{B}|} \leq 4\sum_{i \in [k]} \frac{|B^{(i)} \setminus  B^{(i)}_{1-K\rho}|}{|B^{(i)}|} \leq 50 k K d \rho,\end{align*}
where we used regularity in the last step.\end{proof}

We now deduce a general version of the change of variables lemma. 

\begin{lemma}[Change of variables] \label{chVarLemma} 
Let $B^{(1)}, \dots, B^{(k)}$ be regular Bohr sets of rank at most $d$, and let $Y_1, \dots, Y_\ell \subseteq \cap_{i \in [k]} B^{(i)}_\rho$ be sets. Let $Z_1, \dots, Z_m \subseteq \mathbb{Z}/N\mathbb{Z}$ be arbitrary sets. Let $F \colon (\mathbb{Z}/N\mathbb{Z})^{k+\ell + m} \to \mathbb{D}$ be a function. Let $\nu \in \mathbb{Z}^{[k]\times [\ell]}$ be given. Then
\begin{align*}&\bigg|\exx_{\ssk{x_1 \in B^{(1)}, \dots, x_k \in B^{(k)}\\y_1 \in Y_1, \dots, y_\ell \in Y_\ell\\z_1 \in Z_1, \dots, z_m \in Z_m}} F(x_1, \dots, x_k, y_1, \dots, y_\ell, z_1, \dots, z_m) \\ 
&\hspace{1cm}- \exx_{\ssk{x_1 \in B^{(1)}, \dots, x_k \in B^{(k)}\\y_1 \in Y_1, \dots, y_\ell \in Y_\ell\\z_1 \in Z_1, \dots, z_m \in Z_m}} F\Big(x_1 + \sum_{j \in [\ell]} \nu_{1\,j} y_j, \dots, x_k + \sum_{j \in [\ell]} \nu_{k\,j} y_j, y_1, \dots, y_\ell, z_1, \dots, z_m\Big)\bigg| \leq  50 \rho k \ell d \|\nu\|_\infty.\end{align*}
\end{lemma}

\vspace{\baselineskip}

We refer to variables $y_1, \dots, y_\ell$ as the \emph{dummy variables} as we typically apply the lemma to approximate expressions such as
\[\exx_{a, s,t \in B_\mu, y \in B} f(a, y + a + s) f(a, y + a + t) \approx \exx_{a, s,t \in B_\mu, y \in B} f(a, y + s) f(a, y  + t)\]
when $\mu \ll 1$. In general, we say that we apply the lemma above for the \emph{transformation} $x_i \mapsto x_i + \sum_{j \in [\ell]} \nu_{i\,j} y_j$, for $i \in [k]$.\\

\begin{proof}By triangle inequality, it suffices to prove that
\begin{align*}&\bigg|\exx_{\ssk{x_1 \in B^{(1)}, \dots, x_k \in B^{(k)}\\y_1 \in Y_1, \dots, y_\ell \in Y_\ell}} F(x_1, \dots, x_k, y_1, \dots, y_\ell, z_1, \dots, z_m) \\ 
&\hspace{1cm}- \exx_{\ssk{x_1 \in B^{(1)}, \dots, x_k \in B^{(k)}\\y_1 \in Y_1, \dots, y_\ell \in Y_\ell}} F\Big(x_1 + \sum_{j \in [\ell]} \nu_{1\,j} y_j, \dots, x_k + \sum_{j \in [\ell]} \nu_{k\,j} y_j, y_1, \dots, y_\ell, z_1, \dots, z_m\Big)\bigg|.\end{align*}
is at most $50 \rho k \ell d \|\nu\|_\infty$ for any fixed $z_1 \in Z_1, \dots, z_m \in Z_m$. Applying triangle inequality several times, the expression above is at most
\begin{align*}\sum_{i \in [\ell]} \exx_{\ssk{y_1 \in Y_1, \dots, y_{i-1} \in Y_{i-1}\\y_{i+1} \in Y_{i+1}, \dots, y_\ell \in Y_\ell}} & \bigg|\exx_{\ssk{x_1 \in B^{(1)}, \dots, x_k \in B^{(k)}\\y_i \in Y_i}} F\Big(x_1 + \sum_{j \in [i-1]} \nu_{1\,j} y_j, \dots, x_k + \sum_{j \in [i-1]} \nu_{k\,j} y_j, y_1, \dots, y_\ell, z_1, \dots, z_m\Big)\\
&\hspace{1cm} - \exx_{\ssk{x_1 \in B^{(1)}, \dots, x_k \in B^{(k)}\\y_i \in Y_i}} F\Big(x_1 + \sum_{j \in [i]} \nu_{1\,j} y_j, \dots, x_k + \sum_{j \in [i]} \nu_{k\,j} y_j, y_1, \dots, y_\ell, z_1, \dots, z_m\Big)\bigg|.\end{align*}
The basic case of the lemma applies to give an upper bound of $50 \rho k d \|\nu\|_\infty$ on the term between absolute values, completing the proof.\end{proof}

As a consequence of the change of variables lemma, we are able to average over translates of subsets of small enough dilates of a Bohr set.

\begin{corollary}\label{bohraverage}Let $B$ be a regular Bohr set of rank $d$ and let $f \colon B \to [0,1]$ be a function such that $\ex_{x \in B} f(x) \geq c$. Let $\lambda, \varepsilon > 0$ be such that $\lambda \leq \varepsilon/(200 d)$. Let $S \subseteq B_\lambda$ be a non-empty set. Then there exists a translate $t + S \subseteq B$ such that
\[\exx_{x \in t + S} f(x) \geq c - \varepsilon.\]
\end{corollary}

\begin{proof}By regularity of $B$, provided $\lambda \leq \varepsilon/(24 d)$, we have
\[\exx_{x \in B_{1 - \lambda}} f(x) \geq c - \varepsilon/2.\]
Provided $\lambda \leq \varepsilon/(200 d)$, Lemma~\ref{chVarLemma} gives
\[\exx_{x \in B_{1 - \lambda}, y \in S} f(x + y) \geq \exx_{x \in B_{1 - \lambda}, y \in S} f(x) -\varepsilon/2\geq c - \varepsilon.\]
The corollary follows after averaging over $x \in B_{1-\lambda}$.\end{proof}

\subsection{Almost-periodicity} 

As in the case of Kelley and Meka's proof of Roth's theorem, efficient almost-periodicity for convolutions play a crucial role in our paper. Such results originate in the work of Croot and Sisask~\cite{CSap}, and were used by Sanders~\cite{Sanders}, along with other ideas, to prove quasipolynomial bounds in the Bogolyubov-Ruzsa lemma. These results were further combined by Schoen and Sisask~\cite{SchSisask} to give an even stronger almost-periodicity results. Here we record a variant of their result.

\begin{theorem}[Almost-periodicity]\label{apthm}
Let $\varepsilon > 0$ and let $B^{(1)}$ and $B^{(2)}$ be regular Bohr sets of rank $d$. Let $B^{(2)}$ have radius $\rho$. Suppose that $Y \subseteq B^{(1)}$ and $Z \subseteq B^{(2)}$ are sets with relative densities $\beta = |Y|/|B^{(1)}|$ and $\gamma = |Z|/|B^{(2)}|$. Let $D \subseteq \mathbb{Z}/N\mathbb{Z}$ be an arbitrary set such that $|Y| \leq |D| \leq 2|B^{(1)}|$. Then there exists a further regular Bohr set $B' \subseteq B^{(2)}$ of rank at most $d + d'$ and radius at least $\rho \varepsilon \beta /(24 d^3 d')$ where
\[d' \leq O\Big(\varepsilon^{-4}\log^3(2\beta^{-1})\log(2\gamma^{-1})\Big)\]
such that
\[\Big|\exx_{b \in B', y \in  B^{(1)}, z \in  B^{(2)}} \id_D(z-y + b)\id_Y(y)\id_Z(z) - \exx_{y \in  B^{(1)}, z \in  B^{(2)}} \id_D(z-y)\id_Y(y)\id_Z(z)\Big| \leq \varepsilon \beta\gamma.\]
\end{theorem}

The statement is very similar to Theorem 16 of~\cite{BloomSisask}, the only difference is that we use radius instead of density of Bohr sets in the conclusion.

\begin{proof}We rely on Theorem 5.4 from~\cite{SchSisask}. Set $\lambda = 1/(12d)$. Using the notation of that theorem, we set $A = -Z$, $S =  B^{(2)}_\lambda$, $M = Y$, $L = D$ and $B =  B^{(2)}_\lambda$. The radius of $B$ is thus $\lambda \rho$. The relative density of $S$ inside $B$ is 1 and $\eta = |M|/|L| = |Y|/|D| \geq \beta/2$, while $|Y| \leq |D|$ implies $\eta \leq 1$. Also, $K = |A + S|/|A| \leq \frac{|B^{(2)}_{1 + \lambda}|}{\gamma|B^{(2)}|} \leq (1 + 12 d \lambda) \gamma^{-1} \leq 2\gamma^{-1}$, using regularity of $B^{(2)}_\lambda$. Theorem 5.4 of~\cite{SchSisask} gives the desired Bohr set $B'$ the bound
\[d' \leq O\Big(\varepsilon^{-4}\log^3(2\beta^{-1})\log(2\gamma^{-1})\Big)\]
and shows in particular that 
\[\Big|\sum_{x, y, z \in \mathbb{Z}/N\mathbb{Z}} \id_A(x) \id_M(y) \id_L(z) \frac{1}{|B'|}\id_{B'}(-x - y- z) - \sum_{x, y\in \mathbb{Z}/N\mathbb{Z}} \id_A(x) \id_M(y) \id_L(-x-y)\Big| \leq \varepsilon |A||M|.\]
Returning to our notation, we obtain 
\[\Big|\frac{1}{|B'|}\sum_{b \in B', y, z \in \mathbb{Z}/N\mathbb{Z}} \id_Y(y) \id_Z(z) \id_D(z - y + b) - \sum_{y,z\in \mathbb{Z}/N\mathbb{Z}} \id_Y(y) \id_Z(z) \id_D(z-y)\Big| \leq \varepsilon |Y||Z|.\]
Since $Y \subseteq  B^{(1)}$ and $Z \subseteq B^{(2)}$ we may restrict the variables $y$ and $z$ to $B^{(1)}$ and $B^{(2)}$. By dividing by $|B^{(1)}||B^{(2)}|$, we obtain
\[\Big|\exx_{b \in B', y \in B^{(1)}, z \in B^{(2)}} \id_Y(y) \id_Z(z) \id_D(z - y + b) - \exx_{ y \in B^{(1)}, z \in B^{(2)}} \id_Y(y) \id_Z(z) \id_D(z-y)\Big| \leq \varepsilon \beta\gamma.\qedhere\]
\end{proof}

\subsection{Gowers-H\"older inequalities}

For the purposes of quantifying quasirandomness, we need to consider a family of multilinear forms, that we shall refer to as \emph{vertical segments Gowers inner product}. Let us first recall the notion of grid norms, introduced by Kelley, Lovett and Meka in~\cite{KLMgrid}. For two sets $X$ and $Y$, $k$ and $\ell$ positive even integers\footnote{The condition that $k$ and $\ell$ are even is not essential for the definition; it plays a role in the proof of the Gowers-Cauchy-Schwarz inequality.} and a function $f \colon X \times Y \to \mathbb{R}$, grid norm $\|f\|_{\mathsf{U}_{k, \ell}}$ by the expression
\[\|f\|_{\mathsf{U}_{k, \ell}} = \Big|\exx_{\ssk{x_1, \dots, x_k \in X\\y_1, \dots, y_\ell\in Y}} \prod_{i \in [k], j \in [\ell]} f(x_i, y_j)\Big|^{1/k \ell}.\]

An important property of the norms $\|\cdot\|_{\mathsf{U}_{k, \ell}}$ is that they satisfy a variant of the Gowers-Cauchy-Schwarz inequality, similarly to the original Gowers norms.

\begin{proposition}[Lemma 2.2 in~\cite{FHHK}] Let $f_{i\,j} \colon X \times Y \to \mathbb{R}$ be a function for each $i \in [k], j \in [\ell]$. Then
\[\Big|\exx_{\ssk{x_1, \dots, x_k \in X\\y_1, \dots, y_\ell\in Y}} \prod_{i \in [k], j \in [\ell]} f_{i \, j}(x_i, y_j)\Big| \leq \prod_{i \in [k], j \in [\ell]} \|f_{i \, j}\|_{\mathsf{U}_{k, \ell}}.\]\end{proposition}

In this paper, we need the following special case. Let $f_1, \dots, f_r, g_1, \dots, g_r \colon \mathbb{Z}/N\mathbb{Z} \times \mathbb{Z}/N\mathbb{Z} \to \mathbb{R}$. For two subsets $B$ and $B'$ of $\mathbb{Z}/N\mathbb{Z}$ (which will be Bohr sets later), define the \emph{vertical segments Gowers inner product} $\langle f_1, \dots, f_r, g_1, \dots, g_r \rangle_{\mathsf{VS}_r(B, B')}$ as 
\[\langle f_1, \dots, f_r, g_1, \dots, g_r \rangle_{\mathsf{VS}_r(B, B')} = \exx_{\ssk{y_1, \dots, y_r \in B\\a_1, \dots, a_r, s, s' \in B'}}\prod_{i \in [r]} f_i(a_i, y_i + s) g_i(a_i, y_i + s').\] 
We also define \emph{vertical segments norm} $\|f\|_{\mathsf{VS}_r(B, B')} = \langle f, \dots, f \rangle_{\mathsf{VS}_r(B, B')}^{1/2r}$, where we put $2r$ copies of $f$ as arguments. The relationship with grid norms is the following. Let $X = B' \times B$ and $Y = B'$. Given a function $f \colon \mathbb{Z}/N\mathbb{Z} \times \mathbb{Z}/N\mathbb{Z} \to \mathbb{R}$, consider the function $F \colon X \times Y \to \mathbb{R}$, given by $F\Big((x,y), z\Big) = f(x, y + z)$. Then $\|f\|_{\mathsf{VS}_r(B, B')} = \|F\|_{\mathsf{U}_{r, 2}}$.\\
\indent Owing to this connection, we know that vertical segments Gowers inner product satisfies a suitable Gowers-Cauchy-Schwarz inequality. However, we need slightly more control in our proof, so we record the intermediary steps in the proof of such inequality as a separate proposition. We refer to these inequalities as \emph{Gowers-H\"older inequalities} given that one uses H\"older's inequality instead of Cauchy-Schwarz inequality.

\begin{proposition}[Gowers-H\"older inequality for the vertical segments Gowers inner product]\label{gowholdIneq}Let $B, B' \subseteq \mathbb{Z}/N\mathbb{Z}$ be sets. Let $f_1, \dots, f_r, g_1, \dots, g_r \colon \mathbb{Z}/N\mathbb{Z} \times \mathbb{Z}/N\mathbb{Z} \to \mathbb{R}$ be functions. Then
\begin{itemize}
\item[\textbf{(i)}] whenever $r$ is even,
\[\Big|\langle f_1, \dots, f_r, g_1, \dots, g_r \rangle_{\mathsf{VS}_r(B, B')}\Big|^r \leq \prod_{i \in [r]} \langle f_i, f_i, \dots, f_i, g_i, g_i, \dots, g_i \rangle_{\mathsf{VS}_r(B, B')},\]
(where in the vertical segments Gowers inner products on the right-hand-side we have $r$ copies of $f_i$, followed by $r$ copies of $g_i$),
\item[\textbf{(ii)}] for all $r \in \mathbb{N}$,
\[\Big|\langle f_1, \dots, f_r, g_1, \dots, g_r \rangle_{\mathsf{VS}_r(B, B')}\Big|^2 \leq \langle f_1, \dots, f_r, f_1, \dots, f_r \rangle_{\mathsf{VS}_r(B, B')} \langle g_1, \dots, g_r, g_1, \dots, g_r \rangle_{\mathsf{VS}_r(B, B')}.\]
\end{itemize}
\end{proposition}

\begin{proof}\textbf{Inequality \textbf{(i)}.} Recall that $r$ is even. Thus
\[\Big|\langle f_1, \dots, f_r, g_1, \dots, g_r \rangle_{\mathsf{VS}_r(B, B')}\Big| = \Big|\exx_{s, s' \in B'}\prod_{i \in [r]} \Big(\exx_{a \in B', y \in B} f_i(a, y + s) g_i(a, y + s')\Big)\Big|.\]
Let $F_i(s, s') = \ex_{a \in B', y \in B} f_i(a, y + s) g_i(a, y + s')$. By generalized H\"older's inequality, we get
\[\Big|\langle f_1, \dots, f_r, g_1, \dots, g_r \rangle_{\mathsf{VS}_r(B, B')}\Big| = \Big|\exx_{s, s' \in B'} F_1(s,s') \dots F_r(s,s')\Big| \leq \prod_{i \in [r]} \|F_i\|_{L^r(B' \times B')}.\]
Taking $r$\tss{th} power, we get
\[\|F_i\|_{L^r(B' \times B')}^r = \exx_{s, s' \in B'} \Big|\exx_{a \in B', y \in B} f_i(a, y + s) g_i(a, y + s')\Big|^r = \langle f_i, f_i, \dots, f_i, g_i, g_i, \dots, g_i \rangle_{\mathsf{VS}_r(B, B')},\]
using the fact that $r$ is even.\\
\indent \textbf{Inequality \textbf{(ii)}.} This time, we use Cauchy-Schwarz inequality
\begin{align*}&\Big|\langle f_1, \dots, f_r, g_1, \dots, g_r \rangle_{\mathsf{VS}_r(B, B')}\Big|^2 = \Big|\exx_{\ssk{y_1, \dots, y_r \in B\\a_1, \dots, a_r \in B'}} \Big(\exx_{s \in B'}\prod_{i \in [r]} f_i(a_i, y_i + s)\Big)\Big(\exx_{s \in B'}\prod_{i \in [r]} g_i(a_i, y_i + s)\Big)\Big|^2\\
&\hspace{2cm}\leq \Big(\exx_{\ssk{y_1, \dots, y_r \in B\\a_1, \dots, a_r \in B'}} \Big|\exx_{s \in B'}\prod_{i \in [r]} f_i(a_i, y_i + s)\Big|^2\Big) \Big(\exx_{\ssk{y_1, \dots, y_r \in B\\a_1, \dots, a_r \in B'}} \Big|\exx_{s \in B'}\prod_{i \in [r]} g_i(a_i, y_i + s)\Big|^2\Big)\\
&\hspace{2cm}= \langle f_1, \dots, f_r, f_1, \dots, f_r \rangle_{\mathsf{VS}_r(B, B')} \langle g_1, \dots, g_r, g_1, \dots, g_r \rangle_{\mathsf{VS}_r(B, B')}.\qedhere\end{align*}
\end{proof}

In the analysis of various vertical segments Gowers inner products, we shall also make use of the following inequality due to Kelley and Meka.

\begin{lemma}[Kelley and Meka~\cite{KelleyMeka}, Appendix D, proof of Proposition 5.7]\label{binomialIneq} Let $r' \geq r$ be positive integers and let $\varepsilon > 0$. Then, provided $r' \geq 2\varepsilon^{-1} r$, we have
\[\sum_{\ssk{d\text{ even}\\r \leq d \leq r'}} \binom{r'}{d} \varepsilon^d \geq (1 + \varepsilon/2)^{r'}.\]\end{lemma}

\section{Skew corner-free sets}

Recall that structural hierarchy in the problem of finding skew corners had the form $A \subseteq X \times (t + B) \subseteq (s + B) \times (t + B) \subseteq \mathbb{Z}/N\mathbb{Z} \times \mathbb{Z}/N\mathbb{Z}$ for a Bohr set $B$.\footnote{In the introduction, we discussed the finite vector space case so we had subgroups instead of Bohr sets.} Our argument is based on the density increment argument, where we increase the relative density of $A$ inside $X \times B$ at each step. This comes at the cost of slightly decreasing the relative density of $X$ inside $s + B$ and decreasing the radius of $B$ and increasing its rank somewhat. However, these costs are small enough to allow us to obtain the Behrend-type final bounds. The following proposition articulates the density increment step and is the essence of the argument.

\begin{proposition}\label{densincprop}There exists an absolute constant $C \geq 1$ for which the following holds. Let $B \subseteq \mathbb{Z}/N\mathbb{Z}$ be a regular Bohr set of rank $d$ and radius $\rho$ and let $X \subseteq B$ be a set of relative density $\delta$. Let $A \subseteq X \times B$ be a set of relative density $\alpha$. Suppose that 
\begin{equation}|B_{\mu_0}| \geq 9 \alpha^{-2}\delta^{-2}\label{smallnesscond}\end{equation}
for $\mu_0 = \Big(\frac{\alpha \delta}{2d}\Big)^{C \log^2 2 \delta^{-1}}$.\\
\indent Suppose that $A$ contains no non-trivial skew corners. Then there exist a further regular Bohr set $B' \subseteq B$ and a positive quantity
\[d' \leq C\log^{C}(2\alpha^{-1})\log^{C}(2\delta^{-1})\]
such that $B'$ has rank at most $d + d'$ and radius at least $\rho \exp(-d')/d^3$ and for some translates $x_0 + B'$ and $y_0 + B'$ we have
\[|A \cap ((X \cap (x_0 + B')) \times (y_0 + B'))| \geq (1 + C^{-1}) \alpha |X \cap (x_0 + B')|\,|y_0 + B'|\]
and
\[|X \cap (x_0 + B')| \geq \frac{1}{C}\alpha^C \delta |B'|.\]
\end{proposition}

Let us conclude the main result from the proposition above.

\begin{proof}[Proof of Theorem~\ref{mainthm}] Let $A \subseteq [N] \times [N]$ be a set of density $\alpha$ without skew corners. We need to pass to the cyclic group setting. Without loss of generality, $3| N$. Consider $A \cap [i N/3, (i + 1) N/3) \times [j N/3, (j + 1) N/3)$ for $i, j \in \{0, 1,2\}$. There exists a choice of such a set $A'$ which has relative density at least $\frac{1}{9} \alpha$ inside $[N] \times [N]$. Viewing $A'$ as a subset of $\mathbb{Z}/N\mathbb{Z} \times \mathbb{Z}/N\mathbb{Z}$ instead, $A'$ remains a skew corner-free set.\\

Having passed to the $\mathbb{Z}/N\mathbb{Z} \times \mathbb{Z}/N\mathbb{Z}$ setting, we are ready to perform the density increment argument. Let $C$ be the absolute constant in Proposition~\ref{densincprop}. At $i$\tss{th} step of the iteration, we shall have a regular Bohr set $B^{(i)}$ of rank $d_i$ and radius $\rho_i$ and a set $X^{(i)} \subseteq B^{(i)}$ of relative density $\delta_i$ such that $A'$ has relative density $(1 + C^{-1})^i \alpha$ inside $(x^{(i)} + X^{(i)}) \times (y^{(i)} + B^{(i)})$ for some elements $x^{(i)}, y^{(i)} \in \mathbb{Z}/N\mathbb{Z}$. The following bounds will hold at each step:
\begin{itemize}
\item $\delta_i \geq \alpha^{Ci}C^{-i}$,
\item $d_i \leq C^{C+1} (i + 1)^{C+1} \log^{2C}(2C\alpha^{-1})$,
\item $\rho_i \geq \frac{\exp\big(-C^{C+1} (i + 1)^{C+1} \log^{2C}(2C\alpha^{-1})\big)}{\big(C^{C+1} (i + 1)^{C+1} \log^{2C} (2C\alpha^{-1})\big)^{3i}}$.
\end{itemize}
Initially, we set $X^{(0)} = B^{(0)} = \mathbb{Z}/N\mathbb{Z}$, $d_0 = 0, \rho_0 = 1, \delta_0 = 1$. At each step of procedure we apply Proposition~\ref{densincprop}, as long as the relative density of $A'$ inside the relevant product is at most 1 and condition~\eqref{smallnesscond} holds. It is easy to see that the claimed bounds on $\delta_i, d_i, \rho_i$ hold at each step. Owing to density of $A'$, the procedure terminates after at most $O(\log 2\alpha^{-1})$ steps, and since $A'$ is skew corner-free, it must be the case that condition~\eqref{smallnesscond} fails. Assume that $i$\tss{th} step was the final one. Lemma~\ref{basicBohr} implies the bound
\begin{equation}(\rho_i \mu/2\pi)^{d_i} N \leq 9 \alpha_i^{-2}\delta_i^{-2},\label{finalcond}\end{equation}
for $\mu = (\alpha_i \delta_i/2d_i)^{C \log^2(2\delta_i^{-1})}$. The bounds above allow us to find a $K = O(\log^{O(1)}(2\alpha^{-1}))$ such that 
\[\delta_i^{-1} \leq \exp(-K),\,\, d_i \leq K,\,\, \rho_i^{-1} \leq \exp(-K).\]
These bounds, combined with~\eqref{finalcond}, give
\[N \leq \exp\Big((2K)^{O(1)}\Big) \leq \exp\Big(O(\log^{O(1)}(2\alpha^{-1})\Big)\]
completing the proof.\end{proof}

The remained of the paper is devoted to the proof of Proposition~\ref{densincprop}. Before we begin, let us record here some of the various parameters that will play a role in the proof. We shall use:
\begin{itemize}
    \item constants $c_1 = 2^{-13}, c_2 = 2^{-14}$ and $c_3 = 2^{-12}$,
    \item even positive integers $r$ and $r'$, which are roughly $\log 2\delta^{-1}$,
    \item parameters $\lambda, \mu, \nu$, used as dilate scales of the given Bohr set $B$, which will satisfy
    \[\lambda \approx \Big(\frac{\alpha^4\delta}{32}\Big)^{r^2},\,\, \mu \approx \Big(\frac{\alpha^4\delta}{32}\Big)^{64r} \lambda,\,\, \nu \approx \Big(\frac{\alpha^4 \delta}{8}\Big)^{r'} \mu.\]
\end{itemize}
We shall explain their roles during the proof, but we opted to mention them in advance to help the reader with tracking the quantitative dependencies.

\begin{proof}[Proof of Proposition~\ref{densincprop}] \textbf{Step 1. Regularization of the set $A$.} We first perform a simple regularization step in which we ensure that all columns $\{y \in B \colon (x,y) \in A\}$ of $A$ have roughly the same density in $B$ and also that $X$ has no significant density increments in further Bohr subsets of $B$. The latter property will be used at a few places in \textbf{Step 2}.\\

Fix $c_1 = 2^{-13}$. Let us partition $X = X_1 \cup X_2 \cup X_3$ into three sets according to sizes of the columns indexed by their elements. Namely, let $X_1$ consist of those $x \in X$ such that $|\{y \in B \colon (x,y) \in A\}| \geq (1 + c^2_1) \alpha |B|$, let $X_2$ consist of those $x \in X$ such that $|\{y \in B \colon (x,y) \in A\}| \in [(1-c_1)\alpha|B|, (1 + c^2_1) \alpha |B|)$, and let $X_3 = X \setminus (X_1 \cup X_2)$.\\
\indent If it happens that $|X_1| \geq c^2_1 \alpha |X|$, we may consider $A \cap (X_1 \times B)$, which now has relative density at least $(1 + c_1^2) \alpha$ inside $X_1 \times B$, completing the proof.\\
\indent Now assume the contrary. We then have
\begin{align*}&\sum_{x \in X_2} |\{y \in B \colon (x,y) \in A\}| + \sum_{x \in X_3}|\{y \in B \colon (x,y) \in A\}|\\
&\hspace{2cm}\geq \alpha |X||B| - \sum_{x_1 \in X} |\{y \in B \colon (x,y) \in A\}| \geq \alpha |X| |B| - c^2_1 \alpha |X| |B|.\end{align*}
Subtract $(1 - c_1)\alpha |X_2 \cup X_3| |B|$ from both sides to obtain
\begin{align*} &\sum_{x \in X_2} \Big(|\{y \in B \colon (x,y) \in A\}| - (1 - c_1)\alpha |B|\Big) + \sum_{x \in X_3}\Big(|\{y \in B \colon (x,y) \in A\}| -(1 - c_1)\alpha |B|\Big)\\
&\hspace{2cm}\geq (1 - c_1^2) \alpha |X| |B| - (1 - c_1)\alpha |X_2 \cup X_3||B| \\
&\hspace{2cm}\geq (1 - c_1^2) \alpha |X| |B| - (1 - c_1)\alpha |X||B|\,\, =\,\, (c_1 - c_1^2) \alpha |X||B|.\end{align*}
However, the value of $|\{y \in B \colon (x,y) \in A\}| - (1 - c_1)\alpha |B|$ is at most $(c_1 + c_1^2)\alpha |B|$ when $x \in X_2$ and at most $0$ when $x \in X_3$, from which we deduce $(c_1 + c_1^2)\alpha |B| |X_2| \geq (c_1 - c_1^2) \alpha |X||B|$, implying that $|X_2| \geq \frac{c_1 - c_1^2}{c_1 + c_1^2} |X|$. In particular, $X_2$ has relative density at least $\delta/2$ inside $B$.\\

Let $c_2 = 2^{-14}$,
\begin{equation}r = 2\lceil\log_2(2\delta^{-1})\rceil + 2^{30}\label{rdefn}\end{equation}
(note that $r$ is an even integer) and 
\[\kappa = \frac{1}{2^{20} r^2 d} \Big(\frac{\alpha^4\delta}{32}\Big)^{64r}.\]
Next, we pass to a translate of a dilate of $B$ on which $X_2$ has no density increments by a factor of $1 + c_2$ with respect to further translates of dilates scaled by a factor of at most $\kappa$. Set initially $t_0 = 0$ and $\lambda_0 = 1$. Assume that at $i$\tss{th} step we have translate of a dilate $t_i + B_{\lambda_i}$ such that $B_{\lambda_i}$ is regular, $\lambda_i \geq \kappa^i$ and $|X_2 \cap (t_i + B_{\lambda_i})| \geq (1 + c_2)^i |X_2|/|B|$. Either this translate of a dilate satisfies the required conditions, or we may find further translate of a dilate $t_{i + 1} + B_{\lambda_{i + 1}}$ such that $B_{\lambda_{i + 1}}$ is regular and
\[\frac{|X_2 \cap (t_{i + 1} + B_{\lambda_{i+1}})|}{|B_{\lambda_{i+1}}|} \geq (1 + c_2)\frac{|X_2 \cap (t_i + B_{\lambda_i})|}{|B_{\lambda_i}|}.\]

This procedure terminates after at most $O(\log 2\delta^{-1})$ steps, producing desired translate of a dilate $t + B_\lambda$. Hence, $\lambda \geq \kappa^{O(\log2\delta^{-1})}$ and, as we the relative density of $X_2$ increases at each step, we may assume that $X_2 \cap (t + B_{\lambda})$ has relative density $\delta' \geq \frac{\delta}{2}$ inside $t + B_{\lambda}$. Finally, the relative density inside any $t' + B_{\mu}$ is at most $(1 + c_2) \delta'$ for any $t'$ and $\mu \geq \kappa \lambda$, provided $B_{\mu}$ is regular.\\

We may shift $A$ by $(-t, 0)$, allowing us to assume without loss of generality that $t = 0$. Let us now misuse the notation and write $X$ for the intersection $X_2 \cap B_{\lambda}$, use $\delta$ for density of $X$ relative to $B_\lambda$, which does not affect the form of the final bounds. Hence, we have $A \subseteq X \times B \subseteq B_\lambda \times B$.\\

\noindent\textbf{Step 2. Finding imbalance in $A$.} The goal of this step is to show that $A$ is not quasirandom. Formally, we show that the vertical segments norm of the (suitably defined) balanced function of $A$ is large. Namely, let us write $\vd(x) = \ex_{y \in B} \id_A(x,y)$, which is the relative density of the column of $A$ with coordinate $x$ inside $B$ (and is 0 whenever $x \notin X$), and let us define the \emph{balanced function} of $A$ as $\bal(x,y) = \id_A(x,y) - \vd(x)$. Let us note obvious bounds $\|\vd\|_\infty \leq 1$, since $\vd(x)$ is average of an indicator function, and $\|\bal\|_\infty \leq 1$, since $\bal(x,y) = 1$ or $\bal(x,y) = 1 - a$ for some $a \in [0,1]$.\\
\indent Before exploiting the lack of skew corners, we first need to show some basic bounds concerning simpler configurations. Let $\mu \in [\kappa \lambda, 2\kappa \lambda]$ be a parameter such that $B_\mu$ is regular.

\begin{claim}\label{startingBounds}We have
\[\exx_{\ssk{x \in B_\lambda\\a,s,t \in B_\mu}} \id_X(x + s - t) \vd(x + a)^2 \leq \frac{10}{9} \alpha^2\delta^2\]
and
\begin{align*}\exx_{\ssk{x \in B_\lambda, y \in B\\a,s,t \in B_\mu}} \id_X(x + s - t) \id_A(x + a, y + s)\vd(x + a) \geq &\frac{8}{9} \alpha^2 \delta^2\\
\exx_{\ssk{x \in B_\lambda, y \in B\\a,s,t \in B_\mu}} \id_X(x + s - t) \id_A(x + a, y + a + t)\vd(x + a)\geq &\frac{8}{9} \alpha^2 \delta^2.\end{align*}\end{claim}

\begin{proof}To obtain the first bound, observe that $\vd(x) \leq (1 + c_1^2) \alpha \id_X(x)$ for all $x$, so we have
\[\exx_{\ssk{x \in B_\lambda\\a,s,t \in B_\mu}} \id_X(x + s - t) \vd(x + a)^2 \leq (1 + c_1^2)^2 \alpha^2 \exx_{\ssk{x \in B_\lambda\\a,s,t \in B_\mu}} \id_X(x + s - t)\id_X(x + a).\]
Furthermore, we know that $\ex_{a \in B_\mu} \id_X(x + a) \leq (1 + c_2) \delta \leq \frac{101}{100}\delta$ for all $x$, as otherwise $X$ has a significant density increment on translate of $B_\mu$. On the other hand, since $\mu \leq \frac{\delta}{2^{15} d} \lambda$, the change of variables lemma (Lemma~\ref{chVarLemma}) for transformation $x \mapsto x + s - t$ implies that 
\[\Big|\exx_{\ssk{x \in B_\lambda\\s,t \in B_\mu}} \id_X(x + s - t) - \exx_{\ssk{x \in B_\lambda\\s,t \in B_\mu}} \id_X(x)\Big| \leq 100 d \frac{\mu}{\lambda} \leq \frac{1}{100} \delta\]
and we know that $\ex_{\ssk{x \in B_\lambda\\s,t \in B_\mu}} \id_X(x) = \delta$. The bound follows from these estimates.\\
\indent When it comes to the second bound, the change of variables lemma (Lemma~\ref{chVarLemma}) for transformations $y \mapsto y + s$ and for $y \mapsto y + a + t$, implies that 
\begin{equation}\Big|\exx_{\ssk{x \in B_\lambda, y \in B\\a,s,t \in B_\mu}} \id_X(x + s - t) \id_A(x + a, y + s)\vd(x + a) - \exx_{\ssk{x \in B_\lambda, y \in B\\a,s,t \in B_\mu}} \id_X(x + s - t) \id_A(x + a, y)\vd(x + a)\Big| \leq 50 d \mu\label{basicconfigclaimeq1}\end{equation}
and
\begin{equation}\Big|\exx_{\ssk{x \in B_\lambda, y \in B\\a,s,t \in B_\mu}} \id_X(x + s - t) \id_A(x + a, y + a + t)\vd(x + a) - \exx_{\ssk{x \in B_\lambda, y \in B\\a,s,t \in B_\mu}} \id_X(x + s - t) \id_A(x + a, y)\vd(x + a)\Big| \leq 100 d \mu.\label{basicconfigclaimeq2}\end{equation}
On the other hand,
\begin{align}\exx_{\ssk{x \in B_\lambda, y \in B\\a,s,t \in B_\mu}} \id_X(x + s - t) \id_A(x + a, y)\vd(x + a) = &\exx_{\ssk{x \in B_\lambda\\a,s,t \in B_\mu}} \id_X(x + s - t) \vd(x + a)^2\nonumber\\
\geq &(1 - c_1)^2 \alpha^2 \exx_{\ssk{x \in B_\lambda\\a,s,t \in B_\mu}} \id_X(x + s - t) \id_X(x + a).\label{claim1bound2}\end{align}
Recalling that $X$ has no significant density increments on translates of $B_\mu$, we get bounds
\[\exx_{x \in B_\lambda} \Big|\exx_{s,t \in B_\mu}\id_X(x + s - t)\Big|^2 \leq \exx_{x \in B_\lambda} \Big|\exx_{s \in B_\mu}\Big(\exx_{t \in B_\mu} \id_X(x + s - t)\Big)\Big|^2 \leq (1 + c_2)^2 \delta^2\]
and $\ex_{a \in B_\mu} \id_X(x + a) \leq (1 + c_2) \delta$. By the change of variables lemma (Lemma~\ref{chVarLemma}) for transformation $x \mapsto x + a$, we have 
\[\Big|\exx_{x \in B_\lambda, a \in B_\mu} \id_X(x + a) - \delta \Big| \leq 50 d \frac{\mu}{\lambda}.\]
Hence,
\begin{align*}\exx_{x \in B_\lambda} \Big|\exx_{a \in B_\mu}\id_X(x + a) - \delta\Big|^2 = &\exx_{\ssk{x \in B_\lambda\\a,a' \in B_\mu}} \id_X(x + a) \id_X(x + a') - 2 \delta\exx_{\ssk{x \in B_\lambda\\a,a' \in B_\mu}} \id_X(x + a) + \delta^2 \\
\leq &\exx_{x \in B_\lambda, a \in B_\mu} \id_X(x + a) \Big(\exx_{a' \in B_\mu} \id_X(x + a')\Big) - \delta^2 + 100 d \frac{\mu}{\lambda}\\
\leq &(1 + c_2) \delta\exx_{x \in B_\lambda, a \in B_\mu} \id_X(x + a)  - \delta^2  + 100 d \frac{\mu}{\lambda}\\
\leq &c_2\delta^2  + 200 d \frac{\mu}{\lambda}\leq 2c_2 \delta^2,\end{align*}
using $\mu \leq c_2 \delta^2 \lambda / (200 d)$. By triangle and Cauchy-Schwarz inequalities, we get
\begin{align*}&\Big|\exx_{\ssk{x \in B_\lambda\\a,s,t \in B_\mu}} \id_X(x + s - t) \id_X(x + a) - \delta^2\Big| \leq \Big|\exx_{x \in B_\lambda} \Big(\exx_{s,t \in B_\mu} \id_X(x + s - t)\Big)\Big(\exx_{a \in B_\mu} \id_X(x + a) - \delta\Big)\Big|\\
&\hspace{8cm}+ \delta \Big|\exx_{x \in B_\lambda, s,t \in B_\mu} \id_X(x + s - t)- \delta\Big| \\
&\hspace{2cm}\leq \sqrt{\exx_{x \in B_\lambda} \Big|\exx_{s,t \in B_\mu} \id_X(x + s - t)\Big|^2} \sqrt{\exx_{x \in B_\lambda} \Big|\exx_{a \in B_\mu} \id_X(x + a) - \delta\Big|^2} + 100 d \frac{\mu}{\lambda} \leq 3\sqrt{c_2} \delta^2.\end{align*}
Returning to~\eqref{claim1bound2}, we deduce that
\[\exx_{\ssk{x \in B_\lambda, y \in B\\a,s,t \in B_\mu}} \id_X(x + s - t) \id_A(x + a, y)\vd(x + a)  \geq (1 - c_1)^2(1 - 3\sqrt{c_2})\alpha^2\delta^2 \geq \frac{19}{20}\alpha^2\delta^2.\]
Since $\mu \leq \frac{\alpha^2 \delta^2}{2000 d}$, the claim follows from inequalities~\eqref{basicconfigclaimeq1} and~\eqref{basicconfigclaimeq2}.\end{proof}

\indent Since there are no non-trivial skew corners in $A$, using a slightly artificial parametrization (in order to reduce the number of changes of variables) of skew corners $(x + s - t, y'), (x + a, y + s), (x + a, y + a + t)$, we have
\begin{align*}\exx_{\ssk{x \in B_\lambda, y \in B\\a,s,t \in B_{\mu}}} \id_X(x + s - t) \id_A(x + a, y + s) \id_A(x + a, y + a + t) = &\exx_{\ssk{x \in B_\lambda, y \in B\\a,s,t \in B_{\mu}}}  \id_A(x + a, y + s) \id(a + t - s = 0)\\
&\hspace{2cm} \leq   \exx_{a,s,t \in B_{\mu}} \id(a + t = s) \leq \frac{1}{|B_\mu|}.\end{align*}

Using Claim~\ref{startingBounds}, as long as $|B_\mu| \geq 9\alpha^{-2}\delta^{-2}$, we get
\begin{align*}&\Big|\exx_{\ssk{x \in B_\lambda, y \in B\\a,s,t \in B_\mu}} \id_X(x + s - t) \bal(x + a, y + s) \bal(x + a, y + a + t)\Big|\\
= &\Big| \exx_{\ssk{x \in B_\lambda, y \in B\\a,s,t \in B_\mu}} \id_X(x + s - t) \id_A(x + a, y + s) \id_A(x + a, y + a + t)\\
&\hspace{2cm}- \exx_{\ssk{x \in B_\lambda, y \in B\\a,s,t \in B_\mu}} \id_X(x + s - t) \id_A(x + a, y + s) \vd(x + a)\\
&\hspace{2cm} -\exx_{\ssk{x \in B_\lambda, y \in B\\a,s,t \in B_\mu}} \id_X(x + s - t) \vd(x + a) \id_A(x + a, y + a + t)\\
&\hspace{2cm} + \exx_{\ssk{x \in B_\lambda, y \in B\\a,s,t \in B_\mu}} \id_X(x + s - t) \vd(x + a)^2\Big|\\
\geq & \frac{1}{2} \alpha^2 \delta^2.\end{align*}

Recall that $r = 2\lceil\log_2(2\delta^{-1})\rceil + 2^{30}$, defined in~\eqref{rdefn}, is an even positive integer. Let $\chi \colon B_\lambda \times B_\mu \times B_\mu \to \mathbb{R}$ be given by $\chi(x,s,t) = \id_X(x +s - t)$. H\"{o}lder's inequality for norms $L^{r/(r-1)}$ and $L^r$ on $B_\lambda \times B_\mu \times B_\mu$ gives 
\[\frac{1}{2} \alpha^2 \delta^2 \leq \Big\|\chi\Big\|_{L^{\frac{r}{r-1}}(B_\lambda \times B_\mu \times B_\mu)} \Big(\exx_{\ssk{x \in B_\lambda\\s,t \in B_\mu}} \Big|\exx_{y \in B, a \in B_\mu} \bal(x + a, y + s) \bal(x + a, y + a + t)\Big|^r\Big)^{1/r}.\]
Log-convexity of $L^p$ norms implies $\|\chi\|_{L^{\frac{r}{r-1}}} \leq \|\chi\|_{L^1} \Big(\frac{\|\chi\|_{L^2}}{\|\chi\|_{L^1}}\Big)^{\frac{2}{r}}$. Furthermore, since $\id_X$ only takes values 0 and 1, we have $\|\chi\|_{L^2} = \|\chi\|_{L^1}^{1/2}$. Additionally, change of variables lemma (Lemma~\ref{chVarLemma}) gives $\delta/2 \leq \|\chi\|_{L^1} \leq 2\delta$, since $\mu \leq \frac{\delta}{200d}\lambda$. These facts together imply $\|\chi\|_{L^{\frac{r}{r-1}}} \leq 4\delta$, since $r \geq \log_2(2\delta^{-1})$. After expanding, using the fact that $r$ is even, we obtain

\[\exx_{x \in B_\lambda} \Big(\exx_{\ssk{y_1, \dots, y_r \in B\\s, t, a_1, \dots, a_r \in B_\mu}} \prod_{i \in [r]} \bal(x + a_i, y_i + s) \prod_{i \in [r]} \bal(x + a_i, y_i + a_i + t)\Big) \geq \Big(\frac{\alpha^2 \delta}{8}\Big)^r.\]

By averaging, there is $x_0 \in B_\lambda$ such that

\begin{equation}\exx_{\ssk{y_1, \dots, y_r \in B\\s, t, a_1, \dots, a_r \in B_\mu}} \prod_{i \in [r]} \bal(x_0 + a_i, y_i + s) \prod_{i \in [r]} \bal(x_0 + a_i, y_i + a_i + t) \geq \Big(\frac{\alpha^2 \delta}{8}\Big)^r.\label{holderstep1}\end{equation}

However, the long expression equals
\[\exx_{s,t \in B_\mu} \Big|\exx_{y \in B, a \in B_\mu} \bal(x_0 + a, y + s)\bal(x_0 + a, y + a + t)\Big|^r\]
so by averaging, there are $s_0, t_0 \in B_\mu$ such that 
\[\frac{\alpha^2 \delta}{8} \leq \Big|\exx_{y \in B, a \in B_\mu} \bal(x_0 + a, y + s_0)\bal(x_0 + a, y + a + t_0)\Big| \leq \exx_{a \in B_\mu} \id_X(x_0 + a)\]
so $|X \cap (x_0 + B_\mu)| \geq \frac{\alpha^2 \delta}{8} |B_\mu|$. Let us write $\delta_\mu = \frac{|X \cap (x_0 + B_\mu)|}{|B_\mu|}$, which is the relative density of $X$ inside $x_0 + B_\mu$. Recalling that $X$ has no significant density increments on translates of $B_\mu$, we have thus obtained
\begin{equation}\label{reldensitymu} \frac{\alpha^2 \delta}{8} \leq \delta_\mu \leq (1 + c_2)\delta.\end{equation} 

Applying Cauchy-Schwarz inequality in~\eqref{holderstep1}, we get
\begin{align}\Big(\frac{\alpha^2 \delta}{8}\Big)^{2r} \leq &\bigg|\exx_{\ssk{y_1, \dots, y_r \in B\\a_1, \dots, a_r \in B_\mu}} \Big(\exx_{s \in B_\mu} \prod_{i \in [r]} \bal(x_0 + a_i, y_i + s)\Big) \Big(\exx_{t \in B_\mu} \prod_{i \in [r]} \bal(x_0 + a_i, y_i + a_i + t)\Big)\bigg|^2\nonumber\\
\leq & \bigg(\exx_{\ssk{y_1, \dots, y_r \in B\\a_1, \dots, a_r \in B_\mu}} \Big|\exx_{s \in B_\mu} \prod_{i \in [r]} \bal(x_0 + a_i, y_i + s)\Big|^2\bigg)\nonumber\\
&\hspace{5cm} \bigg(\exx_{\ssk{y_1, \dots, y_r \in B\\a_1, \dots, a_r \in B_\mu}} \Big|\exx_{t \in B_\mu} \prod_{i \in [r]} \bal(x_0 + a_i, y_i + a_i + t)\Big|^2\bigg).\label{holderstep2}\end{align}

It follows from the change of variables lemma (Lemma~\ref{chVarLemma}) for transformation $y_i \mapsto y_i + a_i$, $i \in [r]$, that
\begin{equation}\bigg|\exx_{\ssk{y_1, \dots, y_r \in B\\a_1, \dots, a_r \in B_\mu}} \Big|\exx_{t \in B_\mu} \prod_{i \in [r]} \bal(x_0 + a_i, y_i + a_i + t)\Big|^2 - \exx_{\ssk{y_1, \dots, y_r \in B\\a_1, \dots, a_r \in B_\mu}} \Big|\exx_{t \in B_\mu} \prod_{i \in [r]} \bal(x_0 + a_i, y_i + t)\Big|^2\bigg| \leq 50 \mu r^2 d.\label{holderstep3}\end{equation}
Combining inequalities~\eqref{holderstep2} and~\eqref{holderstep3}, we obtain
\begin{equation}\Big(\frac{\alpha^2 \delta}{16}\Big)^{r} \leq \Big(\frac{\alpha^2 \delta}{8}\Big)^{r} - 50 \mu r^2 d \leq \exx_{\ssk{y_1, \dots, y_r \in B\\a_1, \dots, a_r, s, s' \in B_\mu}}\prod_{i \in [r]} \bal(x_0 + a_i, y_i + s) \bal(x_0 + a_i, y_i + s') = \|T_{(x_0, 0)}\bal\|_{\mathsf{VS}_r(B, B_\mu)}^{2r},\label{disbalanceineq}\end{equation}
where $T_{(x_0, 0)}\bal$ stands for the function $(x,y) \mapsto \bal(x_0 + x, y)$, (i.e. the translation operator applied to $\bal$), as $\mu \leq \frac{1}{100r^2d}\Big(\frac{\alpha^2 \delta}{8}\Big)^{r}$.\\

\noindent\textbf{Step 3. Passing from $\bal$ to $\id_A$.} In this step, we start from the fact that $\|T_{(x_0, 0)}\bal\|_{\mathsf{VS}_r(B, B_\mu)} \geq \Omega(\alpha \sqrt{\delta})$, and we aim to prove $\|T_{(x_0, 0)}\id_A\|_{\mathsf{VS}_{r'}(B, B_\mu)} \geq (1 + \Omega(1))\alpha \sqrt{\delta}$ where $r' = 64r$. Consider the vertical segments Gowers inner product 
\[\langle f_1, \dots, f_{r'}, g_1, \dots, g_{r'} \rangle_{\mathsf{VS}_{r'}(B, B_\mu)}\]
where each $f_i$ and $g_i$ are either $(x, y) \mapsto \bal(x_0 + x, y)$ or $(x,y) \mapsto \vd(x_0 + x)$. Using Gowers-H\"older's inequality (Lemma~\ref{gowholdIneq} \textbf{(i)}), we get
\begin{equation}\langle f_1, \dots, f_{r'}, g_1, \dots, g_{r'} \rangle_{\mathsf{VS}_{r'}(B, B_\mu)}^{r'} \leq \prod_{i \in [r']} \langle f_i, \dots, f_i, g_i, \dots, g_i \rangle_{\mathsf{VS}_r(B, B_\mu)}.\label{gowholdstep}\end{equation}
Suppose that for some $i$ we have $f_i$ and $g_i$ different, so without loss of generality $f_i(x,y) = \bal(x_0 + x, y)$ and $g_i(x,y) = \vd(x_0 + x)$. Then 
\[ \langle f_i, \dots, f_i, g_i, \dots, g_i \rangle_{\mathsf{VS}_{r'}(B, B_\mu)} = \exx_{s, s' \in B_\mu} \Big(\exx_{a \in B_\mu} \Big(\exx_{y \in B}\bal(x_0 + a, y + s)\Big) \vd(x_0 + a) \Big)^{r'}.\]
However, by regularity of $B$, for each $s \in B_\mu$ we have
\begin{align*}\Big|\exx_{y \in B}\bal(x_0 + a, y + s)\Big| = &\Big|\exx_{y \in B}\id_A(x_0 + a, y + s) - \vd(x_0 + a)\Big| \\
=& \Big|\exx_{y \in B}\id_A(x_0 + a, y + s) - \id_A(x_0 + a, y)\Big| \leq |B \Delta s + B|/|B| \leq 24 d \mu,\end{align*}
which, together with~\eqref{gowholdstep} and bounds $\|\bal\|_\infty, \|\vd\|_\infty \leq 1$, implies that
\[\langle f_1, \dots, f_{r'}, g_1, \dots, g_{r'} \rangle_{\mathsf{VS}_{r'}(B, B_\mu)} \leq 24d\mu.\]

We now turn to the case when $f_i = g_i$ holds for all $i \in [r']$. In that case, we always have
\[\langle f_1, \dots, f_{r'}, f_1, \dots, f_{r'} \rangle_{\mathsf{VS}_{r'}(B, B_\mu)} = \exx_{\ssk{y_1, \dots, y_{r'} \in B\\a_1, \dots, a_{r'} \in B_\mu}} \Big|\exx_{s \in B_\mu} \prod_{i \in [r']} f_i(x_0 + a_i, y_i + s)\Big|^2 \geq 0.\]

We need additional estimates, so let us write
\begin{align*}\Pi_k = \langle \underbrace{T_{(x_0, 0)}\bal, \dots, T_{(x_0, 0)}\bal}_k,& \underbrace{T_{(x_0, 0)}\vd, \dots, T_{(x_0, 0)}\vd}_{r' - k},\\
&\hspace{1cm}\underbrace{T_{(x_0, 0)}\bal, \dots, T_{(x_0, 0)}\bal}_{k}, \underbrace{T_{(x_0, 0)}\vd, \dots, T_{(x_0, 0)}\vd}_{r' - k} \rangle_{\mathsf{VS}_{r'}(B, B_\mu)},\end{align*}
which is the vertical segments Gowers inner product for $k$ pairs of $T_{(x_0, 0)}\bal$ terms and $r' -k$ pairs of $T_{(x_0, 0)}\vd$ terms. We showed above that $\Pi_k \geq 0$ for all $k$.\\

Assume that $k$ is even and that $k \geq r$. Note that
\begin{align*}\Pi_k = &\exx_{s, s' \in B_\mu} \Big(\exx_{y \in B, a \in B_\mu} \bal(x_0 + a, y + s)\bal(x_0 + a, y + s')\Big)^k \Big(\exx_{a \in B_\mu} \vd(x_0 + a)^2\Big)^{r' - k}\\
\geq &(1-c_1)^{2r' - 2k} \alpha^{2r' -2k} \exx_{s, s' \in B_\mu} \Big(\exx_{y \in B, a \in B_\mu} \bal(x_0 + a, y + s)\bal(x_0 + a, y + s')\Big)^k \Big(\exx_{a \in B_\mu} \id_X(x_0 + a)\Big)^{r' - k}\end{align*}
as $ \vd(x_0 + a) \geq (1-c_1)\alpha \id_X(x_0 + a)$ for all $a$. Recalling that $\delta_\mu = \ex_{a \in B_\mu} \id_X(x_0 + a)$, this further equals
 \[(1-c_1)^{2r' - 2k} \alpha^{2r' -2k} \delta_\mu^{r' - k} \exx_{s, s' \in B_\mu} \Big(\exx_{y \in B, a \in B_\mu} \bal(x_0 + a, y + s)\bal(x_0 + a, y + s')\Big)^k.\]
By the monotonicity of $L^p$ norms, this expression is at least
\[(1-c_1)^{2r' - 2k} \alpha^{2r' -2k} \delta_\mu^{r' - k} \bigg(\exx_{s, s' \in B_\mu} \Big(\exx_{y \in B, a \in B_\mu} \bal(x_0 + a, y + s)\bal(x_0 + a, y + s')\Big)^r\bigg)^{k/r}.\]
By inequality~\eqref{disbalanceineq}, we get a further lower bound
\[(1-c_1)^{2r' - 2k} \alpha^{2r' -2k} \delta_\mu^{r' - k} \Big(\frac{\alpha^2 \delta}{16}\Big)^{k}\]
which, using~\eqref{reldensitymu}, is at least
\[(1-c_1)^{2r' - 2k} \alpha^{2r' -2k} \delta_\mu^{r' - k} \frac{1}{(1 + c_2)^k}\Big(\frac{\alpha^2 \delta_\mu}{16}\Big)^{k}.\]
Since $(1-c_1)^2/(1 + c_2) \geq 1/2$, we finally get
\[\Pi_k \geq \Big(\frac{1}{32}\Big)^k (1 -c_1)^{2r'} \alpha^{2r'} \delta_\mu^{r'}.\]

Expanding out
\[\|T_{(x_0, 0)}\id_A\|^{2r'}_{\mathsf{VS}_{r'}(B, B_\mu)} =  \langle T_{(x_0, 0)}\bal_A + T_{(x_0, 0)}\vd, \dots, T_{(x_0, 0)}\bal_A + T_{(x_0, 0)}\vd\rangle_{\mathsf{VS}_{r'}(B, B_\mu)},\]
which results in $4^{r'}$ terms, and putting all the above information on vertical segments Gowers inner products together, we obtain 
\begin{align*} \|T_{(x_0, 0)}\id_A\|^{2r'}_{\mathsf{VS}_{r'}(B, B_\mu)} \geq & \sum_{k = 0}^{r'} \binom{r'}{k} \Pi_k \,\, -\,\, 4^{r'} \cdot 24d\mu\\
\geq &\sum_{\ssk{k\text{ even}\\r \leq k \leq r'}} \binom{r'}{k} \Pi_k \,\, -\,\, 4^{r'} \cdot 24d\mu\\
\geq &(1 -c_1)^{2r'} \alpha^{2r'} \delta_\mu^{r'}  \sum_{\ssk{k\text{ even}\\r \leq k \leq r'}} \binom{r'}{k} 32^{-k} \,\, -\,\, 4^{r'} \cdot 24d\mu\\
\geq &(1 -c_1)^{2r'} \alpha^{2r'} \delta_\mu^{r'}  \Big(1 + \frac{1}{64}\Big)^{r'} \,\, -\,\, 4^{r'} \cdot 24d\mu\\
\geq & \Big(1 + \frac{1}{256}\Big)^{r'} \alpha^{2r'} \delta_\mu^{r'},\end{align*}
where we used Lemma~\ref{binomialIneq} in the penultimate step and inequalities $c_1 \leq 2^{-9}$ and $\mu \leq \frac{1}{2^{20}d} \Big(\frac{\alpha^4 \delta}{32}\Big)^{r'}$ in the last one.\\

\noindent\textbf{Step 4. Sifting.} Expanding out $\|T_{(x_0, 0)} \id_A\|^{2r'}_{\mathsf{VS}_{r'}(B, B_\mu)} \geq \Big(1 + \frac{1}{256}\Big)^{r'} \alpha^{2r'} \delta_\mu^{r'}$ we obtain
\[\Big(1 + \frac{1}{256}\Big)^{r'} \alpha^{2r'} \delta_\mu^{r'} \leq \exx_{\ssk{y_1, \dots, y_{r'} \in B\\a_1, \dots, a_{r'}, s, s' \in B_\mu}} \prod_{i \in [r']} \id_A(x_0 + a_i, y_i + s) \id_A(x_0 + a_i, y_i + s').\]
Let us rename the variables $y_i$ to $t_i$, $s$ to $y$ and $s'$ to $y'$. The expression above becomes
\[\Big(1 + \frac{1}{256}\Big)^{r'} \alpha^{2r'} \delta_\mu^{r'} \leq \exx_{\ssk{t_1, \dots, t_{r'} \in B\\a_1, \dots, a_{r'}, y, y' \in B_\mu}} \prod_{i \in [r']} \id_A(x_0 + a_i, t_i + y) \id_A(x_0 + a_i, t_i + y').\]
We may average of $y' \in B_\mu$ to find some $y'_0$ for which
\[\Big(1 + \frac{1}{256}\Big)^{r'} \alpha^{2r'} \delta_\mu^{r'} \leq \exx_{\ssk{t_1, \dots, t_{r'} \in B\\a_1, \dots, a_{r'}, y  \in B_\mu}} \prod_{i \in [r']} \id_A(x_0 + a_i, t_i + y) \id_A(x_0 + a_i, y'_0 + t_i).\] 
Consider the set $D = \{y \in B_{\mu} \colon \ex_{a \in B_\mu, t \in B} \id_A(x_0 + a, t + y)\id_A(x_0 + a, y'_0 + t) \geq (1 + 1/512) \alpha^2\delta_\mu\}$. Let $c_3 = 2^{-12}$. Then
\begin{align*}&\sum_{y \in D} \Big(\exx_{a \in B_\mu, t \in B} \id_A(x_0 + a, t + y)\id_A(x_0 + a, y'_0 + t) \Big)^{r'} \\
&\hspace{4cm}- (1 - c_3) \sum_{y \in B_{\mu}}  \Big(\exx_{a \in B_\mu, t \in B} \id_A(x_0 + a, t + y)\id_A(x_0 + a, y'_0 + t) \Big)^{r'} \\
\hspace{2cm}&\geq c_3 \sum_{y \in B_{\mu}} \Big(\exx_{a \in B_\mu, t \in B} \id_A(x_0 + a, t + y)\id_A(x_0 + a, y'_0 + t) \Big)^{r'} \\
&\hspace{4cm}- \sum_{y \in B_{\mu}\setminus D} \Big(\exx_{a \in B_\mu, t \in B} \id_A(x_0 + a, t + y)\id_A(x_0 + a, y'_0 + t) \Big)^{r'} \\
\hspace{2cm}&\geq c_3\Big(1 + \frac{1}{256}\Big)^{r'} \alpha^{2r'} \delta_\mu^{r'} |B_{\mu}|- \Big(1 + \frac{1}{512}\Big)^{r'} \alpha^{2r'} \delta_\mu^{r'} |B_{\mu}| \\
\hspace{2cm}&\geq \alpha^{2r'} \delta_\mu^{r'}|B_{\mu}|,\end{align*}
since $r' \geq 2^{10}c_3^{-1}$.\\
\indent Upon expanding, we get
\begin{align*}& \exx_{\ssk{a_1, \dots, a_{r'} \in B_\mu\\t_1, \dots, t_{r'} \in B, y \in B_\mu}} \id_D(y) \prod_{i \in [r']} \id_A(x_0 + a_i, t_i + y)\id_A(x_0 + a_i, y'_0 + t_i)\\
&\hspace{2cm}\geq (1 - c_3) \exx_{\ssk{a_1, \dots, a_{r'} \in B_\mu\\t_1, \dots, t_{r'} \in B, y \in B_\mu}}\prod_{i \in [r']} \id_A(x_0 + a_i, t_i + y)\id_A(x_0 + a_i, y'_0 + t_i)  + \alpha^{2r'} \delta_\mu^{r'}.\end{align*}

Let $\nu$ be another positive parameter such that $B_\nu$ is regular and 
\[\frac{1}{2} \leq \frac{\nu}{\frac{1}{2^{40}(r' + 1) d} \Big(\frac{\alpha^4 \delta}{8}\Big)^{r'} \mu} \leq 1.\]
Let us introduce another dummy variable $z$ that ranges over $B_\nu$. Apply the change of variables lemma (Lemma~\ref{chVarLemma}) for transformation $y \mapsto y - z$, $t_i \mapsto t_i + z$, $i \in [r']$, to get
\begin{align*}& \exx_{\ssk{a_1, \dots, a_{r'} \in B_\mu\\t_1, \dots, t_{r'} \in B, y \in B_\mu, z \in B_\nu}} \id_D(y - z) \prod_{i \in [r']} \id_A(x_0 + a_i, t_i + y)\id_A(x_0 + a_i, y'_0 + t_i + z)\\
&\hspace{2cm}\geq (1 - c_3) \exx_{\ssk{a_1, \dots, a_{r'} \in B_\mu\\t_1, \dots, t_{r'} \in B, y \in B_\mu, z \in B_\nu}}\prod_{i \in [r']} \id_A(x_0 + a_i, t_i + y)\id_A(x_0 + a_i, y'_0 + t_i + z) \\
&\hspace{8cm} + \alpha^{2r'} \delta_\mu^{r'} - 50(r' + 1)d\mu^{-1}\nu.\end{align*}
Since $\nu \leq \frac{1}{100(r'+ 1)d} \Big(\frac{\alpha^4 \delta}{8}\Big)^{r'}\mu$, we have $\frac{1}{2}\alpha^{2r'} \delta_\mu^{r'} \geq 50(r'+ 1)d\mu^{-1}\nu$. By averaging, there exists a choice of $a_1, \dots, a_{r'}, t_1, \dots, t_{r'}$ such that
\begin{align*}& \exx_{y \in B_\mu, z \in B_\nu} \id_D(y - z) \prod_{i \in [r']} \id_A(x_0 + a_i, t_i + y)\id_A(x_0 + a_i, y'_0 + t_i + z)\\
&\hspace{2cm}\geq (1 - c_3) \exx_{y \in B_\mu, z \in B_\nu}\prod_{i \in [r']} \id_A(x_0 + a_i, t_i + y)\id_A(x_0 + a_i, y'_0 + t_i + z)  + \frac{1}{2}\alpha^{2r'} \delta_\mu^{r'}.\end{align*}

Define $Y = \{y \in B_\mu \colon (\forall i \in [r'])\, (x_0 + a_i, y + t_i) \in A\}$ and $Z =  \{z \in B_\nu \colon (\forall i \in [r']) \,(x_0 + a_i, y'_0 + t_i + z) \in A\}$. In terms of $Y$ and $Z$, we obtain the inequality
\begin{align*}&\exx_{y \in B_\mu, z \in B_\nu} \id_D(y - z) \id_Y(y) \id_Z(z)  \geq (1 - c_3)\exx_{y \in B_\mu, z \in B_\nu} \id_Y(y) \id_Z(z)  + \frac{1}{2}\alpha^{2r'} \delta_\mu^{r'}.\end{align*}
Let $\beta = |Y|/|B_\mu|$ and $\gamma = |Z| / |B_\nu|$. We get
\begin{equation}\exx_{y \in B_\mu, z \in B_\nu} \id_D(y - z) \id_Y(y) \id_Z(z) \geq (1-c_3) \beta\gamma\label{densityEQN}\end{equation}
and
\begin{equation}\beta, \gamma, \frac{|D|}{|B_\mu|} \geq \frac{1}{2}\alpha^{2r'} \delta_\mu^{r'}.\label{densityEQN2}\end{equation}

\noindent\textbf{Step 5. Using almost-periodicity.} In this step, we apply Theorem~\ref{apthm}. In the notation of that theorem, we shall set $B^{(1)} = B_\mu$ and $B^{(2)} = B_\nu$ so the condition $|D| \leq 2|B^{(1)}|$ holds immediately. However, due to other technical condition that requires $|Y| \leq |D|$, we modify $Y$ slightly. If $|Y| \,> |D|$, choose a subset $Y'\subseteq Y$ uniformly at random among subsets of size $|D|$. Write $\beta'$ for density of $Y'$ inside $B_\mu$, thus $\beta' = |D|/|B_\mu|$. Using linearity of expectation, we get
\begin{align*}&\exx_{Y'} \bigg( \sum_{y \in B_\mu, z \in B_\nu} \id_D(y - z) \id_{Y'}(y)\id_{Z}(z)  -  (1-c_3) \beta'\gamma|B_\mu||B_\nu|\bigg) \\
&\hspace{2cm} =  \sum_{y  \in B_\mu, z \in B_\nu} \id_D(y - z) \id_{Y}(y)\id_{Z}(z) \mathbb{P}(y \in Y') -  (1-c_3) \beta'\gamma|B_\mu||B_\nu|\\
&\hspace{2cm} = \frac{|Y'|}{|Y|}\Big( \sum_{y  \in B_\mu, z \in B_\nu}\id_D(y - z) \id_{Y}(y)\id_{Z}(z) -  (1-c_3) \beta\gamma|B_\mu||B_\nu|\Big) \geq 0.\end{align*}
Hence, there exists a choice of a subset $Y'$ of size $|D|$ such that $\ex_{y\in B_\mu, z \in B_\nu} \id_D(y - z) \id_{Y'}(y)\id_{Z}(z)  \geq (1-c_3) \beta'\gamma$. We misuse the notation and write $Y$ for this subset and $\beta$ for its density relative to $B_\mu$. Note that the lower bound on $\beta$ does not change, as we had the same lower bound~\eqref{densityEQN2} on $\frac{|D|}{|B_\mu|}$.\\
\indent Theorem~\ref{apthm} (with $\varepsilon = c_3$) provides us with a further regular Bohr set $B' \subseteq B_\nu$ of rank at most $d + d'$ and radius at least $\nu\rho c_3 \beta/(24d^3d')$, where 
\[d' \leq O\Big(\log^3(2\beta^{-1})\log(2\gamma^{-1})\Big)\]
and such that
\[\exx_{y \in B_\mu,z \in B_\nu, w \in B'} \id_D(y - z + w) \id_Y(y) \id_Z(z) \geq (1 - 2c_3) \beta\gamma.\]

Returning to the definition of the set $D$, we get
\[\exx_{\ssk{a \in B_\mu, b \in B\\y \in B_\mu,z \in B_\nu, w \in B'}} \id_A(x_0 + a, y - z + w + b)\id_A(x_0 + a, y'_0 + b) \id_Y(y) \id_Z(z) \geq (1 + 2^{-10}) \alpha^2\beta\gamma \delta_\mu.\]

\noindent\textbf{Step 6. Completing the proof.} By averaging, there are $y_0 \in Y, z_0 \in Z$ such that
\[\exx_{\ssk{a \in B_\mu, b \in B\\w \in B'}} \id_A(x_0 + a, y_0 - z_0 + w + b) \id_A(x_0 + a, y'_0 + b) \geq (1 + 2^{-10}) \alpha^2 \delta_\mu.\]

Let $G$ be the set of $b \in B$ such that $\exx_{a \in B_\mu}\id_A(x_0 + a, y'_0 +  b) \geq 2^{-11}\alpha^2 \delta_\mu$. Then
\[\frac{1}{|B|} \sum_{b \in G} \exx_{a \in B_\mu, w \in B'} \id_A(x_0 + a, y_0 - z_0 + w + b) \id_A(x_0 + a, y'_0 + b) \geq (1 + 2^{-11}) \alpha^2 \delta_\mu.\]

Note also that
\begin{align*}\frac{1}{|B|}\sum_{b \in G} \exx_{a \in B_\mu, w \in B'} \id_A(x_0 + a, y'_0 + b) \leq &\exx_{a \in B_\mu, b \in B} \id_A(x_0 + a, y'_0 + b)\\
= &\exx_{a \in B_\mu} \id_X(x_0 + a) \Big(\exx_{b \in B}\id_A(x_0 + a, y'_0 + b)\Big) \leq (1 + c_1) \alpha\delta_\mu.\end{align*}
Hence, as $c_1 \leq 2^{-13}$, 
\[\frac{1}{|B|} \sum_{b \in G} \exx_{a \in B_\mu, w \in B'} \id_A(x_0 + a, y_0 - z_0 + w + b) \id_A(x_0 + a, y'_0 + b) \geq (1 + 2^{-12}) \alpha \frac{1}{|B|}\sum_{b \in G} \exx_{a \in B_\mu, w \in B'} \id_A(x_0 + a, y'_0 + b).\]

By averaging, there exists $b_0 \in G$ such that
\[\exx_{a \in B_\mu, w \in B'} \id_A(x_0 + a, y_0 - z_0 + b_0 + w) \id_A(x_0 + a, y'_0 + b_0) \geq (1 + 2^{-12}) \alpha \exx_{a \in B_\mu} \id_A(x_0 + a, y'_0 + b_0).\]

Define the set $\tilde{X} = \{a \in x_0 + B_\mu \colon (a,y'_0 + b_0) \in A\}$. Owing to definition of $G$, its relative density $\delta'$ inside $x_0 + B_\mu$ satisfies $\delta' \geq 2^{-11}\alpha^2 \delta_\mu \geq 2^{-14} \alpha^4 \delta$. Using the new notation, the inequality above becomes
\[\exx_{a \in B_\mu, w \in B'}\id_A(x_0 + a, y_0 - z_0 + b_0 + w)  \id_{\tilde{X}}(x_0 + a) \geq (1 + 2^{-12}) \alpha \delta'.\]

Write $w_0 = y_0 - z_0 + b_0$. Let $\tilde{X}'$ be the set of all $x \in \tilde{X}$ such that $\ex_{w \in B'} \id_A(x, w_0 + w) \geq (1 + 2^{-13}) \alpha$. Then 
\[\exx_{a \in B_\mu, w \in B'}\id_A(x_0 + a, w_0 + w)  \id_{\tilde{X}'}(x_0 + a) \geq 2^{-13} \alpha \delta',\]
so in particular $|\tilde{X}'| \geq 2^{-13} \alpha \delta' |B_\mu|$.\\

It remains to ensure that translates of the same Bohr set in vertical and horizontal directions. Recall that $B' \subseteq B_\nu$. Using $\nu \leq \frac{2^{-28}\alpha^5\delta}{200d}\mu$, we may apply Corollary~\ref{bohraverage} to find a translate $t_0 + B' \subseteq x_0 + B_\mu$ such that $|\tilde{X}' \cap (t_0 + B')| \geq 2^{-14} \alpha \delta' |t_0 + B'|$. Write $\tilde{X}'' = \tilde{X}' \cap (t_0 + B')$ and let $\tilde{\delta}$ be $|\tilde{X}''| / |t_0 + B'|$. Finally, we conclude that the relative density of $A \cap (\tilde{X}'' \times (w_0 + B'))$ inside $\tilde{X}'' \times (w_0 + B')$ is precisely
\[\exx_{a \in \tilde{X}'', w \in w_0 + B'} \id_A(a, w) = \frac{|B'|}{|\tilde{X}''|} \exx_{a \in B', w \in B'} \id_A(t_0 + a, w_0 + w) \id_{\tilde{X}''}(t_0 + a) \geq (1 + 2^{-13}) \alpha.\]
Hence, we have obtained the desired density increment inside $\tilde{X}'' \times (w_0 + B') \subseteq (t_0 + B') \times (w_0 + B')$, as required.\end{proof}

\end{document}